\documentclass{amsart}
\usepackage[utf8]{inputenc}

\usepackage{amssymb}
\usepackage{enumerate}
\usepackage{stmaryrd}
\usepackage{mathrsfs}
\usepackage[hidelinks]{hyperref}
\usepackage[cmtip,all]{xy}

\newcommand{\CC}{\mathbb{C}}
\newcommand{\PP}{\mathbb{P}}
\newcommand{\OO}{\mathcal{O}}
\newcommand{\ZZ}{\mathbb{Z}}

\newcommand{\ch}{\operatorname{ch}}

\newcommand{\ext}{\mathrm{ext}}

\newcommand{\Quot}{\mathrm{Quot}}
\newcommand{\Spec}{\operatorname{Spec}}

\newcommand{\Gr}{\mathrm{Gr}}

\newcommand{\Mor}{\mathrm{Mor}}
\newcommand{\inj}{\hookrightarrow}
\newcommand{\surj}{\twoheadrightarrow}
\newcommand{\im}{\mathrm{im}}
\newcommand{\ev}{\mathrm{ev}}

\newtheorem{thm}{Theorem}[subsection]
\newtheorem{lem}[thm]{Lemma}
\newtheorem{prop}[thm]{Proposition}
\newtheorem{cor}[thm]{Corollary}
\newtheorem*{main}{Main Theorem}
\newtheorem*{maincor}{Main Corollary}

\theoremstyle{definition}
\newtheorem{defn}[thm]{Definition}
\newtheorem{ex}[thm]{Example}
\newtheorem{strategy}[thm]{Strategy}

\theoremstyle{remark}
\newtheorem{rmk}[thm]{Remark}

\title[Weighted TQFT for Quot schemes on curves]{A weighted topological quantum field theory for Quot schemes on curves}

\author{Thomas Goller} \thanks{The author was partially supported by the NSF RTG grant DMS-1246989.}
\address{School of Mathematics\\
  Korea Institute for Advanced Study\\
  85 Hoegiro Dongdaemun-gu, Seoul 02455, Republic of Korea}
\email{goller@kias.re.kr}

\begin{document}

\begin{abstract}
We study Quot schemes of vector bundles on algebraic curves.
Marian and Oprea gave a description of a topological quantum field theory (TQFT) studied by Witten in terms of intersection numbers on Quot schemes of trivial bundles. Since these Quot schemes can have the wrong dimension, virtual classes are required. But Quot schemes of general vector bundles always have the right dimension. Using the degree of the general vector bundle as an additional parameter, we construct a weighted TQFT containing both Witten's TQFT and the small quantum cohomology TQFT of the Grassmannian. This weighted TQFT is completely geometric (no virtual classes are needed), can be explicitly computed, and recovers known formulas enumerating the points of finite Quot schemes.
\end{abstract}

\maketitle

\section{Introduction}\label{s:intro}
Let $C$ be a smooth projective curve of genus $g$ over $\CC$. Fixing a coherent sheaf $V$ and a cohomology class $\sigma$ on $X$, let $\Quot(\sigma,V)$ denote the Grothendieck Quot scheme, which is a projective scheme whose closed points parametrize short exact sequences $0 \to E \to V \to F \to 0$ of coherent sheaves in which the Chern character $\ch(E)$ is $\sigma$. We can think of $\sigma$ as an ordered pair $(r,-e)$ encoding the rank $r$ and the degree $-e$ of $E$.

Throughout the paper, we fix positive integers $r$ and $s$, use $r$ for the rank of the subbundles, and let $V$ be a vector bundle on $C$ of rank $r+s$. This allows us to use the simpler notation $Q_{e,V} = \Quot\big((r,-e),V\big)$ for the Quot scheme. Denoting the degree of $V$ by $d \in \ZZ$, the expected dimension of $Q_{e,V}$ is
\[
	\mathrm{expdim}(Q_{e,V}) = rd + (r+s)e - rs(g-1)
\]
or 0 if this number is negative.


When $V = \OO_C^{r+s}$, we often write $Q_{e,C}$ instead of $Q_{e,\OO_C^{r+s}}$. In the particular case $e=0$, the Quot scheme $Q_{0,C}$ parametrizes short exact sequences of the form $0 \to \OO_C^r \to \OO_C^{r+s} \to \OO_C^s \to 0$ and is isomorphic to the Grassmannian $G=\Gr(r,\CC^{r+s})$ of $r$-planes in $\CC^{r+s}$. The Grassmannian contains Schubert varieties $W_{\vec{a}} \subset G$ indexed by partitions $\vec{a}$, whose intersection theory is described using the Schubert cycles $\sigma_{\vec{a}} = [W_{\vec{a}}]$ in the cohomology ring $H^*(G,\CC)$. A more general intersection theory can be studied on any $Q_{e,V}$; we denote the analogous ``Schubert varieties'' by $\overline{W}_{\vec{a}}(q)$ (where $q$ is a point on $C$) and their cohomology classes, which are independent of $q$, by $\overline{\sigma}_{\vec{a}}$. We use the notation $\int_{Q_{e,V}} \alpha$, where $\alpha$ is a cohomology class, for intersection numbers on the Quot scheme.

\subsection{First motivation}\label{ss:motivation1}

Let $C$ be a curve of genus $g \ge 2$ and $V$ be a general stable vector bundle of rank $r+s$. It is known that the Quot scheme parametrizing the rank-$r$ subbundles of $V$ of maximal degree is smooth of the expected dimension (see for instance \cite{Hi88}, \cite{RT99}, \cite{LN02}, \cite{Hol04}). It follows that $Q_{e,V}$ is a finite collection of reduced points whenever its expected dimension is zero. Enumerating the points of these finite Quot schemes produces interesting formulas. For instance, when $V$ has degree $s(g-1)$, Marian and Oprea use the formula
\[
	\# Q_{0,V} = V_g^{r,s} := \sum_{I \subset \{ 1,\dots,r+s \}, |I|=r} \: \prod_{(j,k) \in I \times \bar{I}} \left| 2 \sin \pi \tfrac{j-k}{r+s} \right|^{g-1}
\]
(here $\bar{I}$ denotes the complement of $I$ in $\{ \, 1,\dots,r+s \, \}$), which exhibits the number of points of the Quot scheme as a Verlinde number, to prove strange duality on curves (\cite{MO07}). Holla computes a general formula, which counts arbitrary finite Quot schemes, using Gromov-Witten invariants of the Grassmannian (\cite{Hol04}).

\subsection{Second motivation}\label{ss:motivation2}

There is a topological quantum field theory (TQFT), studied by Witten (\cite{Wit95}), which is described in standard mathematical language by Marian and Oprea (\cite{MO10}). A TQFT consists of a vector space $H$ and a collection of linear maps
\[
	F(g)_m^n \colon H^{\otimes m} \to H^{\otimes n}\qquad (g,m,n \ge 0)
\]
satisfying composition relations related to gluing topological surfaces along boundary circles. In fact, by cutting surfaces into pairs of pants, the entire collection of maps is determined by the product $F(0)_2^1$ and the pairing $F(0)_2^0$.

For Witten's TQFT, the vector space is the cohomology ring $H^*(G,\CC)$ of the Grassmannian $G = \Gr(r,\CC^{r+s})$, which has a basis given by the Schubert cycles $\sigma_{\vec{a}}$. The maps can be defined by specifying the coefficients of the $F(g)_m^n$ in this basis as intersection numbers on Quot schemes. Let $C$ be a smooth curve of genus $g$, let $V = \OO_C^{r+s}$, and compile the Quot schemes as $Q_C = \bigsqcup_{e \in \ZZ} Q_{e,C}$. Then define
\[
	F(g)_m^n \colon \sigma_{\vec{a}_1} \otimes \cdots \otimes \sigma_{\vec{a}_m} \mapsto \sum_{\vec{b}_1,\dots,\vec{b}_n} \left( \int_{[Q_C]^{\text{vir}}} \bar{\sigma}_{\underline{\vec{a}}} \cup \bar{\sigma}_{\underline{\vec{b}}} \cup \bar{\sigma}_{1^r}^{r(g+n)+s} \right) \, \sigma_{\vec{b}_1^c} \otimes \cdots \otimes \sigma_{\vec{b}_n^c},
\]
where we use the notation $\bar{\sigma}_{\underline{\vec{a}}} = \bar{\sigma}_{\vec{a}_1} \cup \cdots \cup \bar{\sigma}_{\vec{a}_m}$ and similarly $\bar{\sigma}_{\underline{\vec{b}}} = \bar{\sigma}_{\vec{b}_1} \cup \cdots \cup \bar{\sigma}_{\vec{b}_n}$. A careful definition of the virtual classes and the proof that these maps satisfy the relations of a TQFT can be found in \cite{MO10}.

More concretely, the product $F(0)_2^1$ turns out to be the (small) quantum product, denoted by $*$, and the pairing is
\[
	F(0)_2^0 \colon \sigma_{\vec{a}} \otimes \sigma_{\vec{b}} \mapsto \langle \sigma_{\vec{a}} * \sigma_{\vec{b}},\sigma_{s^r} \rangle,
\]
where $\langle \,,\, \rangle$ denotes the Poincar\'{e} pairing and $\sigma_{s^r}$ is the class of a point. Remarkably, the ``closed invariants'' $F(g)_0^0$, which can be viewed as numbers, are exactly the Verlinde numbers $V_g^{r,s}$, which hints at a connection between Witten's TQFT and the finite Quot schemes in \S\ref{ss:motivation1}.

\subsection{Synthesis: a weighted TQFT}\label{ss:synthesis}

The coefficients of the maps in Witten's TQFT are defined using the intersection theory of Quot schemes of trivial bundles, but these Quot schemes are not always of the expected dimension, which explains the need for virtual classes. On the other hand, Marian and Oprea's result suggests that Quot schemes of general bundles on curves may exhibit better behavior. Thus one can ask whether replacing the trivial bundles in the definition of $F(g)_m^n$ by general vector bundles of degree 0 can avoid the need for virtual classes.

Once the bundles are no longer assumed to be trivial, it is natural to relax the degree 0 assumption as well. This introduces the degree of the general vector bundle $V$ as an additional parameter, which leads to the definition of a weighted TQFT, in which the extra parameter is additive under composition of maps. Moreover, we can attempt to use this parameter to explain the appearance of $\overline{\sigma}_{1^r}$ in the integrand in the definition of $F(g)_m^n$ since we saw in \S\ref{ss:motivation1} that the corresponding Schubert varieties of type $\overline{W}_{1^r}$ arise naturally in the presence of elementary modifications.

Motivated by these ideas, let $C$ be a curve of genus $g$, let $d \in \mathbb{Z}$, and let $V$ be a general stable vector bundle on $C$ of rank $r+s$ and degree $d$.\footnote{We use the word ``stable'' loosely to include balanced vector bundles in the case of $\PP^1$ and semistable vector bundles in genus 1. See \S\ref{ss:elementary_modification}.} Set $Q_V = \bigsqcup_{e \in \ZZ} Q_{e,V}$. Let $H = H^*(G,\CC)$ as above and define a collection of linear maps $F(g|d)_m^n \colon H^{\otimes m} \to H^{\otimes n}$ for all $g,m,n \ge 0$ and $d \in \ZZ$ by
\[
	F(g|d)_m^n \colon \sigma_{\vec{a}_1} \otimes \cdots \otimes \sigma_{\vec{a}_m} \mapsto \sum_{\vec{b}_1,\dots,\vec{b}_n} \left( \int_{Q_V} \bar{\sigma}_{\underline{\vec{a}}} \cup \bar{\sigma}_{\underline{\vec{b}}} \right) \, \sigma_{\vec{b}_1^c} \otimes \cdots \otimes \sigma_{\vec{b}_n^c}.
\]
We do not need virtual classes because, as we will show, the $Q_{e,V}$ are all of the expected dimension. The main purpose of this paper is to prove the following:
\begin{main}\label{t:wtqft}
\begin{enumerate}[(a)]
\item The $F(g|d)_m^n$ defined above are the maps of a weighted TQFT.
\item This weighted TQFT contains both the quantum cohomology TQFT and Witten's TQFT.
\end{enumerate}
\end{main}


The collection of maps with $d=0$ form a usual TQFT, which we show is the quantum cohomology TQFT of the Grassmannian: the product $F(0|0)_2^1$ is the quantum product $*$ and the pairing $F(0|0)_2^0$ is the Poincar\'{e} pairing $\langle \,,\,\rangle$. The weighted TQFT is determined from this degree 0 TQFT by the degree-lowering operator $F(0|{-}1)_1^1$, which is quantum multiplication by $\sigma_{1^r}$, and the degree-raising operator $F(0|1)_1^1$, which is quantum multiplication by $\sigma_s$. Moreover, the weighted TQFT recovers the maps in Witten's TQFT as
\[
	F(g)_m^n = F(g|{-}r(g{+}n){-}s)_m^n = F(g|s(g{-}1{+}n))_m^n,
\]
where the second equality holds because of the identity $\sigma_{1^r}^{r+s} = 1$ in quantum cohomology (note that $r(g+n)+s$ is exactly the power of $\bar{\sigma}_{1^r}$ appearing in the definition of $F(g)_m^n$). In particular,
\[
	V_g^{r,s} = F(g)_0^0 = F(g|s(g{-}1))_0^0
\]
explains the connection to the formula used by Marian and Oprea to count finite Quot schemes: $F(g|s(g{-}1))_0^0$ is the intersection number associated to the fundamental class on $\sqcup_e Q_{e,V}$, where $V$ is a general bundle of degree $s(g-1)$. This number counts the points of the finite Quot scheme $Q_{0,V}$ since the Quot schemes for other values of $e$ are either empty or positive-dimensional.

In the proof of the Main Theorem, we extend the result mentioned in \S\ref{ss:motivation1} by showing that all Quot schemes of general bundles with expected dimension 0 are finite and reduced, including the cases $g=0$ and $g=1$. We use the structure of the weighted TQFT to recover the Verlinde formula above as well as Holla's formula counting the points of any finite Quot scheme of a general vector bundle (\cite{Hol04}, Theorem 4.2):

\begin{maincor} Let $C$ be a curve of genus $g$, let $e = r'\gamma$ for any $0 \le \gamma < \gcd(r,s)$, and let $V$ be a general vector bundle of degree $d = -(r'+s')\gamma + s(g-1)$ on $C$. Then $Q_{e,V}$ is finite and reduced and its points are counted by
\[
	\# Q_{r'\gamma,V} = (r+s)^{r(g-1)} \sum_I \left( \frac{\sigma_{1^r}(\zeta^I)^{(r'+s')\gamma}}{\mathrm{Vand}(\zeta^I)}\right)^{g-1}.
\]
\end{maincor}
Here $I$ ranges over the $r$-element subsets of $\{\, 1,\dots, r+s \,\}$, $\zeta^I$ is an $r$-tuple of powers of $\zeta := e^{2\pi i/(r+s)}$ determined by $I$, $\sigma_{1^r}(\zeta^I)$ is the product of those powers of $\zeta$, and $\mathrm{Vand}(\zeta^I)$ is the product of all differences of those powers of $\zeta$. This notation is explained in \S\ref{ss:qctqft}.

\begin{proof}[Sketch of proof of Main Theorem (a)] Since any composition of morphisms can be viewed as gluing one pair of boundary circles at a time, it suffices to prove the relations for gluing one pair of boundary circles.

Given a smooth curve $C'$ of genus $g$ and a vector bundle $V'$ of rank $r+s$ and degree $d$, there is a degeneration of $C'$ into a reducible nodal curve $C$ with two smooth components $C_1$ and $C_2$ glued at points $p_1$ and $p_2$ to produce a simple node. The genera $g_1$ and $g_2$ of the components sum to $g$. The vector bundle $V'$ also degenerates, producing a vector bundle $V$ on $C$ that can be described by gluing vector bundles $V_i$ on $C_i$ by an isomorphism $V_1(p_1) \simeq V_2(p_2)$ of their fibers over the node. The degrees $d_i$ of the $V_i$ satisfy $d_1+d_2=d$, which reflects the additivity of the weight in the weighted TQFT under composition. The existence of a relative Quot scheme over the base of the deformation guarantees that the intersection numbers on $Q_{e,V'}$ agree with the intersection numbers on $Q_{e,V}$, namely
\[
	\int_{Q_{e,V'}} \bar{\sigma}_{\underline{\vec{a}}_1} \cup \bar{\sigma}_{\underline{\vec{a}}_2}
	= \int_{Q_{e,V}} \bar{\sigma}_{\underline{\vec{a}}_1} \cup \bar{\sigma}_{\underline{\vec{a}}_2},
\]
where we think of the Schubert cycles $\bar{\sigma}_{\underline{\vec{a}}_i}$ on $Q_{e,V'}$ as being based at points that degenerate to lie on $C_i$ (this subtlety turns out to be unnecessary).

The remaining work is to compute the integral on the nodal curve $C$ in terms of intersection numbers on its components $C_i$. The Quot scheme on $C$ can be described by gluing subsheaves (or quotients) on the $C_i$ along their fibers at the points $p_i$. Roughly, for each partition $e = e_1 + e_2$, the isomorphism $V_1(p_1) \simeq V_2(p_2)$ allows us to think of the codomain of
\[
	\ev_{p_1,p_2} \colon Q_{e_1,V_1} \times Q_{e_2,V_2} \dashrightarrow G \times G
\]
as two copies of the same Grassmannian, and the pairs of quotients that can be glued are the preimage $\Delta_{e_1,e_2}$ of the diagonal. In fact, each $\Delta_{e_1,e_2}$ with $e = e_1 + e_2$ embeds as a dense open set in a top-dimensional component of $Q_{e,V}$, and every top-dimensional component of $Q_{e,V}$ corresponds to such a partition of $e$. Recalling that the cohomology class of the diagonal in $G \times G$ is $\sum_{\vec{b}} \sigma_{\vec{b}} \otimes \sigma_{\vec{b}^c}$, we thus obtain the formula
\[
	\int_{Q_{e,V}} \bar{\sigma}_{\underline{\vec{a}}_1} \cup \bar{\sigma}_{\underline{\vec{a}}_2}
	= \sum_{e_1 + e_2 = e} \sum_{\vec{b}} \left( \int_{Q_{e_1,V_1}}\bar{\sigma}_{\underline{\vec{a}}_1} \cup \bar{\sigma}_{\vec{b}} \right) \left( \int_{Q_{e_2,V_2}}\bar{\sigma}_{\underline{\vec{a}}_2} \cup \bar{\sigma}_{\vec{b}^c} \right).
\]
Combining this formula with the above equality of integrals and summing over all $e$ produces the gluing relation corresponding to identifying boundary circles on two morphisms in the weighted TQFT.
\end{proof}

\subsection*{Plan of paper}
\S \ref{s:background1} gives a brief introduction to the quantum cohomology of Grassmannians and topological quantum field theories. \S \ref{s:explicit_tqfts} assumes Main Theorem (a) and uses previous work on the quantum cohomology TQFT to explicitly describe Witten's TQFT, the weighted TQFT, and the connections between these TQFTs, culminating in proofs of Main Theorem (b) and the Main Corollary.

The remainder of the paper is devoted to developing the intersection theory on Quot schemes required to justify the steps in the sketch of the proof of Main Theorem (a) given above. \S \ref{s:background2} summarizes some preliminaries on Quot schemes, including a description of the recursive structure of the boundary. \S \ref{s:quot_trivial} proves that Quot schemes of trivial bundles have a nice intersection theory when the subbundles have sufficiently negative degree. \S \ref{s:quot_general} shows that Quot schemes of general bundles \emph{all} have a nice intersection theory by embedding them in Quot schemes of trivial bundles. \S \ref{s:quot_nodal} extends the nice intersection theory to curves with a single node by relating the Quot scheme on the nodal curve to the Quot schemes on the smooth components, and concludes the proof of Main Theorem (a).

\subsection*{Summary of good properties of Quot schemes}

The Quot scheme is guaranteed to have nice properties in the following cases:
\begin{enumerate}
\item $\PP^1$, $V$ trivial, all $e$: $Q_{e,\PP^1}$ is irreducible, smooth, and of the expected dimension, and $U_{e,V}$ is open and dense (\cite{Gro95}).
\item $C$ smooth, $V$ arbitrary, $e \gg 0$: $Q_{e,V}$ is irreducible, generically smooth, and of the expected dimension, and $U_{e,V}$ is open and dense (\cite{PR03}).
\item $C$ smooth, $V$ very general, all $e$: $Q_{e,V}$ is equidimensional, generically smooth, and of the expected dimension, and $U_{e,V}$ is open and dense (Proposition \ref{p:quot_general}).
\item $C$ nodal, ($V$ trivial, $e \gg 0$) or ($V$ very general, all $e$): every top-dimensional component of $Q_{e,V}$ is generically smooth and of the expected dimension, and $U_{e,V}$ is open and dense in the union of top-dimensional components (Proposition \ref{p:quot_nodal}).
\end{enumerate}

In each of these cases, the intersections $W$ of general Schubert varieties are proper, $W \cap U_{e,V}$ is dense in $W$, and top intersections $W$ are finite, reduced, and contained in $U_{e,V}$ (see, respectively, \cite{Ber97} and Propositions \ref{p:quot_trivial_props}, \ref{p:quot_general}, \ref{p:quot_nodal}).

\subsection*{Acknowledgments} The author would like to thank Aaron Bertram for providing the vision that led to the weighted TQFT as well as guidance during countless discussions. The author is grateful to Nicolas Perrin and Daewoong Cheong for suggesting useful references.

\section{Background on quantum cohomology and TQFTs}\label{s:background1}
\subsection{Quantum Schubert calculus}\label{ss:quantum_schubert_calculus}

Let $G = \Gr(r,\CC^{r+s})$ denote the Grassmannian of $r$-planes in $\CC^{r+s}$, which is a projective variety of dimension $rs$. Let $V_{\bullet}$ denote a full flag $0 = V_0 \subsetneq V_1 \subsetneq V_2 \subsetneq \cdots \subsetneq V_{r+s}=\CC^{r+s}$ of linear subspaces of $\CC^{r+s}$. We call an $r$-tuple of integers $\vec{a} = (a_1,\dots,a_r)$ satisfying $s \ge a_1 \ge \cdots \ge a_r \ge 0$ a \emph{partition}. Define the \emph{Schubert variety}
\[
	W_{\vec{a}}(V_{\bullet}) := \{\, [E \subset \CC^{r+s}] \in G \mid \text{$\dim(E \cap V_{s+i-a_i}) \ge i$ for all $1 \le i \le r$} \, \},
\]
which has complex codimension in $G$ equal to the size $|\vec{a}| = \sum a_i$ of the partition. Let $\sigma_{\vec{a}} = [W_{\vec{a}}(V_{\bullet})]$ denote the associated \emph{Schubert cycle} in the cohomology ring $H^*(G,\CC)$; we drop $V_{\bullet}$ from the notation because the cohomology class of a Schubert variety does not depend on the flag.  The set $\{\, \sigma_{\vec{a}} \mid \text{$\vec{a}$ a partition as above} \, \}$ is a basis of $H^*(G,\CC)$ as a vector space.

In the basis of Schubert cycles, the \emph{Poincar\'{e} pairing} on $H^*(G,\CC)$ is
\[
	\langle \sigma_{\vec{a}},\sigma_{\vec{b}} \rangle = \# \left( W_{\vec{a}}(V_\bullet) \cap W_{\vec{b}}(V_\bullet') \right) = \int_G \sigma_{\vec{a}} \cup \sigma_{\vec{b}},
\]
where $V_\bullet$ and $V_\bullet'$ are general full flags and the number of intersection points is considered to be 0 if the intersection is infinite. A simple exercise shows that
\[
	\langle \sigma_{\vec{a}},\sigma_{\vec{b}} \rangle = \begin{cases} 1 & \text{if $\vec{b}=\vec{a}^c$;} \\
	0 & \text{otherwise} \end{cases}
\]
where $\vec{a}^c = (s-a_r,s-a_{r-1},\dots,s-a_1)$ denotes the partition complementary to $\vec{a}$. 

We extend the Poincar\'{e} pairing to all cohomology classes by linearity and to more than two arguments by setting
\[
	\langle \sigma_{\vec{a}_1},\dots,\sigma_{\vec{a}_N} \rangle := \int_G \sigma_{\vec{a}_1} \cup \cdots \cup \sigma_{\vec{a}_N}.
\]
By the above fact, we can write
\[
	\sigma_{\vec{a}_1} \cup \cdots \cup \sigma_{\vec{a}_N} = \sum_{\vec{b}} \langle \sigma_{\vec{a}_1},\dots, \sigma_{\vec{a}_N}, \sigma_{\vec{b}} \rangle\, \sigma_{\vec{b}^c}.
\]
There are classical formulas of Giambelli and Pieri for computing arbitrary cup products in the Schubert basis.

We define a quantum deformation of the cup product by introducing a new integer parameter $e \ge 0$. There is a smooth quasiprojective variety $\Mor_e(\PP^1,G)$ of dimension $(r+s)e + rs$ whose closed points are
\[
	\Mor_e(\PP^1,G) = \{\, \text{$\phi \colon \PP^1 \to G$ holomorphic} \mid \deg \phi = e \, \}.
\]
Given a choice of distinct points $p_1,\dots,p_N \in \PP^1$ and general Schubert varieties $W_{\vec{a}_1}$, $\dots$, $W_{\vec{a}_N}$, define the \emph{Gromov-Witten number} $\langle W_{\vec{a}_1},\dots,W_{\vec{a}_N} \rangle_e$ as
\[
	\# \, \{ \, \phi \in \Mor_e(\PP^1,G) \mid \text{$\phi(p_i) \in W_{\vec{a}_i}$ for all $1 \le i \le N$} \, \}
\]
when this number is finite and zero otherwise. Requiring $\phi(p_i) \in W_{\vec{a}_i}$ is a codimension-$|\vec{a}_i|$ condition in $\Mor_e(\PP^1,G)$, hence a necessary condition for $\langle W_{\vec{a}_1},\dots,W_{\vec{a}_N} \rangle_e$ to be nonzero is that the total size of the partitions matches the dimension of $\Mor_e(\PP^1,G)$, namely $\sum |\vec{a}_i| = (r+s)e+rs$.

\begin{rmk}\label{r:gw_independent} As the notation suggests, the Gromov-Witten numbers do not depend on the choice of points or the general full flags used to define the Schubert varieties. To prove this, one redefines them as follows. By the universal property of the Grassmannian, pulling back the universal subbundle on $G$ under a morphism $\phi \colon \PP^1 \to G$ defines a bijection
\[
	\Mor_e(\PP^1,G) \to \{ \, E \subset \OO_{\PP^1}^{r+s} \mid \text{$E$ is a rank $r$ vector subbundle of degree $-e$} \,\}.
\]
The Quot scheme $Q_{e,\PP^1} := \Quot\big( (r,-e), \OO_{\PP^1}^{r+s} \big)$ compactifies $\Mor_e(\PP^1,G)$ by allowing inclusions $E \subset \OO_{\PP^1}^{r+s}$ that drop rank at finitely many points of $\PP^1$. Given a point $p \in \PP^1$ and a Schubert variety $W_{\vec{a}}(V_{\bullet})$ in $G$, we define a subvariety $\overline{W}_{\vec{a}}(p,V_{\bullet})$ of $Q_{e,\PP^1}$ by taking the closure of the set of $\phi \in \Mor_e(\PP^1,G)$ satisfying $\phi(p) \in W_{\vec{a}}(V_{\bullet})$. Since the cohomology class $\bar{\sigma}_{\vec{a}}$ of $\overline{W}_{\vec{a}}(p,V_{\bullet})$ is independent of $p$ and $V_{\bullet}$, so are the numbers 
\[
	\int_{Q_{e,\PP^1}} \bar{\sigma}_{\vec{a}_1} \cup \cdots \cup \bar{\sigma}_{\vec{a}_N},
\]
which coincide with the Gromov-Witten numbers defined above.
\end{rmk}

\begin{ex}[$e=0$]\label{ex:gw_deg0} Since degree 0 maps to the Grassmannian are constant, $\Mor_e(\PP^1,G) = Q_{e,\PP^1} \simeq G$. Thus $\langle W_{\vec{a}_1},\dots,W_{\vec{a}_N} \rangle_0 = \langle \sigma_{\vec{a}_1},\dots,\sigma_{\vec{a}_N} \rangle.$
\end{ex}

\begin{ex} $\langle W_{\vec{a}},W_{\vec{b}} \rangle_e = 0$ for all $e > 0$. Since the automorphism group of $\PP^1$ acts 3-transitively, imposing conditions at only two points of $\PP^1$ results in an infinite collection of maps.
\end{ex}

We now use the Gromov-Witten numbers to define a ``small quantum deformation'' of the cup product, called the \emph{quantum product}:
\[
	\sigma_{\vec{a}_1} * \cdots * \sigma_{\vec{a}_N} := \sum_{e \ge 0} q^e \left( \sum_{\vec{b}} \langle W_{\vec{a}_1},\dots,W_{\vec{a}_N},W_{\vec{b}} \rangle_e \, \sigma_{\vec{b}^c} \right).
\]
By Example \ref{ex:gw_deg0}, the $e=0$ term is $\sigma_{\vec{a}_1} \cup \cdots \cup \sigma_{\vec{a}_N}$. Here $q$ is a formal variable of degree $r+s$ that makes the product homogeneous, which we will usually omit (by setting $q=1$). The sum is finite since $e \gg 0$ (namely $(r+s)e + rs > (N+1)rs$) makes it impossible for an intersection of $N+1$ Schubert varieties in $Q_{e,\PP^1}$ to be zero-dimensional. The quantum product is clearly commutative, but it takes a good deal of work to show it is associative (\cite{RT95}, \cite{KM97}). Since $*$ is associative, all Gromov-Witten numbers (and hence all quantum products) are determined by the three-point numbers $\langle W_{\vec{a}},W_{\vec{b}},W_{\vec{c}} \rangle_e$. Quantum products can be computed using quantum versions of the Giambelli and Pieri formulas (\cite{Ber97}).

\begin{rmk} A convenient way of writing the quantum product is
\[
	\sigma_{\vec{a}_1}*\cdots*\sigma_{\vec{a}_N} = \sum_{\vec{b}} \left( \int_{Q_{\PP^1}} \bar{\sigma}_{\underline{\vec{a}}} \cup \bar{\sigma}_{\vec{b}} \right) \sigma_{\vec{b}^c},
\]
where $Q_{\PP^1} = \sqcup_e Q_{e,\PP^1}$. The associativity of the quantum product is equivalent to the identity
\[
	\int_{Q_{\PP^1}} \bar{\sigma}_{\underline{\vec{a}}} \cup \bar{\sigma}_{\vec{b}} \cup \bar{\sigma}_{\vec{c}} = \int_{Q_{\PP^1}} \bar{\sigma}_{\underline{\vec{a}}} \cup \overline{\sigma_{\vec{b}} * \sigma_{\vec{c}}},
\]
where $\overline{\sigma_{\vec{b}} * \sigma_{\vec{c}}}$ means writing $\sigma_{\vec{b}} * \sigma_{\vec{c}}$ in the Schubert basis and applying ``bar'' to each Schubert cycle. For the purpose of computing integrals with Schubert cycles on $Q_{\PP^1}$, $\text{``bar''} \circ * = \cup \circ \text{``bar''}$.
\end{rmk}

\begin{ex}\label{ex:special_properties} The quantum Pieri's formula implies that $\sigma_{1^r}$ and $\sigma_s$ have the following special properties:
\begin{enumerate}
\item $\sigma_{1^r} * \sigma_{a_1,\dots,a_r} = \begin{cases} \sigma_{a_1+1,a_2+1,\dots,a_r+1} & \text{if $a_1 < r$;} \\ \sigma_{a_2,\dots,a_r,0} & \text{if $a_1=r$} \end{cases}$;
\item $\sigma_{1^r} * \sigma_s = 1$;
\item $\sigma_{1^r}^{r+s} = 1$.
\end{enumerate}
Combining (3) and the associativity of $*$ yields the identity
\[
	\int_{Q_{\PP^1}} \bar{\sigma}_{\underline{\vec{a}}} \cup \bar{\sigma}_{1^r}^{r+s} = \int_{Q_{\PP^1}} \bar{\sigma}_{\underline{\vec{a}}}.
\]
\end{ex}

Equipping the vector space $H^*(G,\CC)$ with the quantum product instead of the cup product yields the \emph{small quantum cohomology ring} $QH^*(G)$ (other authors choose to remember the variable $q$). If we also consider the Poincar\'{e} pairing, which is compatible with the quantum product in the sense that
\[
	\langle \sigma_{\vec{a}} * \sigma_{\vec{b}}, \sigma_{\vec{c}} \rangle = \sum_{e \ge 0} \langle W_{\vec{a}},W_{\vec{b}},W_{\vec{c}}  \rangle_e =  \langle \sigma_{\vec{a}}, \sigma_{\vec{b}} * \sigma_{\vec{c}} \rangle,
\]
then $QH^*(G)$ has the structure of a Frobenius algebra. A geometric way to visualize this structure is as a topological quantum field theory, which we discuss in the next subsection.

\subsection{Topological quantum field theories}\label{ss:tqfts}

The terse description of topological quantum field theories in this section is based on a nicer exposition with motivation and pictures in \cite{Ren05}.

A (two-dimensional) \emph{topological quantum field theory} (TQFT) is a functor of tensor categories
\[
	F \colon \mathrm{2Cob} \to \mathrm{Vect}_\CC.
\]
The category $\text{2Cob}$ is composed of
\begin{enumerate}[(a)]
\item objects: finite disjoint unions of oriented circles;
\item morphisms: equivalence classes of oriented cobordisms, which are oriented topological surfaces with oriented boundary circles;
\item composition: concatenation of cobordisms by gluing boundary circles on one surface to boundary circles on another surface;
\item tensor structure: disjoint union.
\end{enumerate}
Let $H$ denote the image $F(S^1)$, which is a vector space. Since $F$ is a functor of tensor categories, the image of $n$ disjoint copies of $S^1$ is $H^{\otimes n}$ and the image of the empty union of circles is the base field $\CC$. The oriented topological genus $g$ surface with $m+n$ boundary circles, denoted by $\Sigma(g)_m^n$ and viewed as a cobordism between the disjoint union of $m$ circles and the disjoint union of $n$ circles, gets mapped by $F$ to a linear transformation $F(g)_m^n \colon H^{\otimes m} \to H^{\otimes n}$. Thus each surface with boundary specifies an algebraic structure on the tensor powers of $H$ and these algebraic structures must satisfy a large number of relations coming from gluing the topological surfaces.
\begin{ex} \begin{enumerate}[(a)]
\item The cylinder $\Sigma(0)_1^1$ is the identity under concatenation of cobordisms, hence $F(0)_1^1 \colon H \to H$ is the identity map.
\item The closed torus $\Sigma(1)_0^0$ can be obtained by gluing the ends of the cylinder. Gluing the ends corresponds to taking the trace, so $F(1)_0^0 \colon \CC \to \CC$ is determined by $1 \mapsto \dim H$.
\item The ``pair of pants'' $\Sigma(0)_2^1$ defines a product $F(0)_2^1 \colon H \otimes H \to H$ on $H$.
\item The ``cap'' $\Sigma(0)_0^1$ composed with the pair of pants produces the cylinder, so the map $F(0)_0^1 \colon \CC \to H$ picks out the multiplicative identity in $H$.
\item The ``macaroni'' $\Sigma(0)_2^0$ corresponds to a pairing $F(0)_2^0 \colon H \otimes H \to \CC$.
\item Powers of $\Sigma(1)_1^1$ can be used to produce every $\Sigma(g)_1^1$. We call $F(1)_1^1$ the \emph{genus-addition operator}.
\item We think of $F(g)_0^0$, the closed invariants of the TQFT, as complex numbers. Since $\Sigma(g)_0^0$ can be obtained by concatenating $g-1$ copies of $\Sigma(1)_1^1$ and gluing the ends, we see that $F(g)_0^0$ can be computed as the trace of the $(g-1)$th power of the genus-addition operator.
\end{enumerate}
\end{ex}

By decomposing surfaces into pairs of pants, we can write every morphism as a composition of genus 0 surfaces with $\le 3$ boundary circles. In fact, $F$ is determined by $F(0)_2^1$, which defines a product $H \otimes H \to H$, and $F(0)_2^0$, which defines a pairing $H \otimes H \to \CC$. The key idea is that the pairing determines an isomorphism $\phi \colon H \simeq H^*$ sending each orthonormal basis to its dual basis, which can be used to ``flip'' the linear maps in the TQFT: to get $F(g)_n^m$, simply take the transpose of $F(g)_m^n$ and identify each $H^*$ with $H$ using the pairing. Moreover, the composition relations on the TQFT imply that this product and pairing must satisfy the compatibility condition of a Frobenius algebra. Thus a TQFT determines a Frobenius algebra and vice versa.

\begin{ex}[Quantum cohomology TQFT] Equipping the small quantum cohomology ring $QH^*(G)$ with the Poincar\'{e} pairing defines a Frobenius algebra, which determines a TQFT that we call the \emph{quantum cohomology TQFT}. A detailed description of this TQFT is given in \S\ref{ss:qctqft}.
\end{ex}

\begin{ex}[Witten's TQFT]\label{ex:witten_tqft} Equipping $QH^*(G)$ with a different pairing
\[
	\langle \sigma_{\vec{a}},\sigma_{\vec{b}}  \rangle_{\text{W}} := \langle \sigma_{\vec{a}}*\sigma_{\vec{b}},\sigma_{s^r} \rangle
\]
produces a different TQFT that we call \emph{Witten's TQFT}. We show in \S\ref{ss:Wtqft} that this description of Witten's TQFT coincides with the description given in \S\ref{ss:motivation2} and that the closed invariants $F(g)_0^0$ of Witten's TQFT are the Verlinde numbers $V_g^{r,s}$ introduced in \S\ref{ss:motivation1}. Thus the Verlinde numbers can be computed by taking the trace of the $(g-1)$th power of the genus addition operator (whose eigenvalues are the products of sines in the formula for $V_g^{r,s}$). 
\end{ex}

\begin{rmk} To emphasize the connection between TQFTs and Frobenius algebras and simplify notation, we usually write $v \cdot w$ instead of $F(0)_2^1(v \otimes w)$ and $\langle v,w \rangle$ instead of $F(0)_2^0(v \otimes w)$. However, there are several advantages to interpreting Frobenius algebras as TQFTs. First, some of the TQFTs we consider in the next sections are related to algebraic curves, with the cobordisms encoding information about curves with marked points. The gluing relations reflect what happens to that information when the curve degenerates to a nodal curve. Second, the TQFT emphasizes different features of the data than the Frobenius algebra does. In particular, the closed invariants of a TQFT seem obscure from the Frobenius algebra point of view, but they often encode interesting enumerative data, which the TQFT is set up to compute.
\end{rmk}

Choosing a basis $\{ f_i \}$ of $H$ that is orthonormal with respect to the pairing simplifies computations in the TQFT. For instance, the genus-addition operator can be computed as the partial composition of the copairing $F(0)_0^2$ and the product $F(0)_3^1$. Since the pairing is determined by $f_i \otimes f_j \mapsto \langle f_i,f_j \rangle$, the copairing is determined by $1 \mapsto \sum_{i,j} \langle f_i,f_j \rangle f_i \otimes f_j = \sum_i f_i \otimes f_i$. Composing with $F(0)_3^1$, we see that $F(1)_1^1$ is multiplication by $\sum_i f_i^2$.

Computations can be further simplified if the TQFT is \emph{semisimple}, namely if there is an orthonormal basis $\{ e_i \}$ of $H$ such that $e_i^2 = \lambda_i e_i$ for some $0 \ne \lambda_i \in \CC$ and $e_i \cdot e_j = 0$ for $i \ne j$. We call $\{ e_i \}$ a \emph{semisimple basis} and $\{ \, \lambda_i \, \}$ the \emph{semisimple values} of the TQFT. For a semisimple TQFT, the multiplicative unit is $\sum_i \lambda_i^{-1} e_i$ and the genus-addition operator is $\cdot \sum_i \lambda_i e_i$, which in the semisimple basis is the diagonal matrix $\mathrm{diag}(\lambda_i^2)$. Thus, for a semisimple TQFT,
\[
	F(g)_0^0 = \sum_i \lambda_i^{2(g-1)}.
\]
It is known that the quantum cohomology TQFT is semisimple (\cite{Abr00}); we exhibit a semisimple basis in \S\ref{ss:qctqft}.

It can be useful to define variations of a TQFT that encode richer data. The morphisms in a \emph{weighted TQFT} are written as $F(g|d)_m^n$, where $d$ is a weight that is additive under composition. The collection of morphisms with $d=0$ is a usual TQFT. The key extra data in a weighted TQFT are the cylinders with weight $\pm 1$, which can be used to obtain any $F(g|d)_m^n$ from the weight-0 linear maps $F(g|0)_m^n$. Indeed, the data encoded by a weighted TQFT are equivalent to a Frobenius algebra equipped with an invertible operator satisfying some relations.

\begin{ex}[The weighted TQFT] We can construct a weighted TQFT containing the quantum cohomology TQFT as the $d=0$ slice and Witten's TQFT as a ``diagonal'' slice. Let the $d=0$ morphisms be exactly as in the quantum cohomology TQFT, and add the data of the invertible operator $* \sigma_{1^r}$, which we view as the cylinder with weight $-1$, and its inverse $*\sigma_s$, which we view as the cylinder of weight $1$. We show in \S\ref{ss:wtqft} that this description of a weighted TQFT agrees with the description in \S\ref{ss:synthesis} and that the linear map $F(g|s(g{-}1+n))_m^n$ in this weighted TQFT coincides with $F(g)_m^n$ from Witten's TQFT. In particular, the Verlinde numbers occur in the weighted TQFT as the closed invariants $F(g|s(g-1))_0^0$.
\end{ex}

\section{Explicit descriptions of the three TQFTs}\label{s:explicit_tqfts}
In this section, we give a detailed description of the quantum cohomology TQFT. We then give similar descriptions of Witten's TQFT and the weighted TQFT introduced in \S\ref{ss:wtqft} by interpreting the maps in terms of quantum cohomology (as was foreshadowed in \S\ref{ss:tqfts}).

\subsection{The quantum cohomology TQFT} \label{ss:qctqft}

Set $G=\Gr(r,\CC^{r+s})$ and let $QH^*(G)$ denote the small quantum cohomology ring, which is obtained by equipping the vector space $H^*(G,\CC)$ with the quantum product $*$ instead of the cup product.

As explained in \cite{Ber94}, a set of generators and relations for $QH^*(G)$ as a ring can be described as follows. As generators, we take the Chern classes $x_1,x_2,\dots,x_r$ of the dual of the universal subbundle on $G$, which coincide with the Schubert cycles $\sigma_1,\sigma_{1,1},\dots,\sigma_{1^r}$. We can view the $x_i$ as elementary symmetric polynomials in the Chern roots $q_1,\dots,q_r$. Set $W = \sum_i \frac{1}{r+s+1} q_i^{r+s+1}$. Then the relations for the cohomology ring $H^*(G,\CC)$ are given by $\partial W/\partial x_i$. The quantum cohomology ring is obtained by replacing the relation $\partial W/\partial x_1$ by $\partial W/\partial x_1 + (-1)^r$.

It is often useful to consider the scheme $\Spec QH^*(G)$, which is a collection of $\binom{r+s}{r}$ reduced points. The $x_i$-coordinates of these points can be described by letting the $q_i$ vary over each choice of $r$ distinct $(r+s)$th roots of $(-1)^{r+1}$. Fix $\zeta = e^{2\pi i/(r+s)}$. Let $I \subset \{ \, 1,2,\dots,r+s \, \}$ denote a subset of cardinality $r$. Write $\zeta^I$ for the point obtained by choosing the roots $\{\, \zeta^{j} \mid j \in I \,\}$ in the case when $r$ is odd or $\{\, \zeta^{j + 1/2} \mid j \in I \,\}$ in the case when $r$ is even. Then the evaluation of $\sigma_{\vec{a}}$ (viewed as a polynomial in the $x_i$) at the point $\zeta^I$, which we denote by $\sigma_{\vec{a}}(\zeta^I)$, coincides with evaluating the Schur polynomial associated to the partition $\vec{a}$ at the roots defining $\zeta^I$. A quantity that often shows up in formulas is the square of the Vandermonde norm, which is defined as
\[
\mathrm{Vand}(\zeta^I) := \left\|\prod_{j,k \in I,\, j \ne k} (\zeta^j - \zeta^k) \right\|
\]
(or 1 if $|I|=1$). We do not have to worry about the parity of $r$ since the factors of $\zeta^{1/2}$ do not affect the norm.

The following proposition is essentially contained in \cite{Rie01}.
\begin{prop}[Concrete description of the quantum cohomology TQFT]\label{p:qctqft}{\ }
\begin{enumerate}[(a)]
\item The basis $\{ \, \sigma_I := \sum_{\vec{a}} \overline{\sigma_{\vec{a}}(\zeta^I)} \sigma_{\vec{a}} \, \}$ has the property that
\[
	\sigma_{\vec{a}} * \sigma_I = \sigma_{\vec{a}}(\zeta^I) \sigma_I
\]
for all $\vec{a}$ and satisfies
\[
	\sigma_I * \sigma_J = \begin{cases}
    	a_I \sigma_I & \text{if $I = J$;}\\
        0 & \text{if $I \ne J$}
    \end{cases}
    \quad \text{for }
    a_I = \frac{(r+s)^r}{\mathrm{Vand}(\zeta^I)}.
\]

\item A semisimple basis is $\{ \, e_I := \left( \sigma_{s^r}(\zeta^I) / a_I \right)^{1/2} \sigma_I \, \}$, with semisimple values $\{ \, \lambda_I = (a_I \sigma_{s^r}(\zeta^I))^{1/2} \, \}$.

\item The genus-addition operator is $* \sum_I \sigma_{s^r}(\zeta^I) \sigma_I$ and its $e_I$-eigenvalue is $\lambda_I^2 = a_I \sigma_{s^r}(\zeta^I)$.
\end{enumerate}
\end{prop}

\begin{proof} The orthogonality formula
\[
	\sum_{\vec{a}} \overline{\sigma_{\vec{a}}(\zeta^I)} \sigma_{\vec{a}}(\zeta^J) = \delta_{I,J} \frac{(r+s)^r}{\mathrm{Vand}(\zeta^I)}
\]
(\cite{Rie01} Prop. 4.3 (3)) implies that $\sigma_I$ is $a_I$ times the characteristic function of the point $\zeta^I$. This proves (a). To get a semisimple basis $\{ e_I = \sigma_I / \|\sigma_I\| \}$, we compute
\[
	\| \sigma_I \|^2 = \langle \sigma_I,\sigma_I \rangle = \langle \sigma_I^2, 1 \rangle = a_I \langle \sigma_I, 1 \rangle = a_I \, \overline{\sigma_{s^r}(\zeta^I)} = a_I/\sigma_{s^r}(\zeta^I),
\]
where the last equality holds since $\sigma_{s^r}(\zeta^I)$ is a product of roots of $(-1)^{r+1}$ and thus has norm 1. Choosing a square root for each $I$ gives the expression for $e_I$ in $(b)$. Moreover,
\[
	e_I^2 = (\sigma_{s^r}(\zeta^I)/a_I )^{1/2} a_I \, e_I = (a_I \, \sigma_{s^r}(\zeta^I))^{1/2} \, e_I
\]
gives the formula for $\lambda_I$. The genus-addition operator is computed as $\sum_I \lambda_I e_I$.
\end{proof}

\begin{ex}[$\PP^s$] When $r=1$, the Grassmannian is isomorphic to $\PP^s$, $x_1$ is the hyperplane class, the cohomology ring is $H^*(\PP^s,\CC) \simeq \CC[x_1]/(x_1^{s+1})$, and the quantum cohomology ring is $QH^*(\PP^s) \simeq \CC[x_1]/(x_1^{s+1}-1)$. The points of $\Spec QH^*(\PP^s)$ are $x_1=q_1=\zeta^j$ for $1 \le j \le s+1$ and $\zeta = e^{2\pi i/(r+s)}$. The Schur basis is $\{ \, 1,\sigma_1,\dots,\sigma_s \, \}$, where $\sigma_j = x_1^j$, and the semisimple basis is 
\[
	\left\{ \, e_j = \sqrt{\frac{\zeta^{js}}{s+1}} \sum_{k=0}^{s} \zeta^{-jk} \sigma_k \;\bigg| \; 1 \le j \le s+1 \, \right\}.
\]
\end{ex}

\subsection{Explicit description of Witten's TQFT}\label{ss:Wtqft}

Recall from \S\ref{ss:motivation2} that the maps in Witten's TQFT are defined as
\[
	F(g)_m^n \colon \sigma_{\vec{a}_1} \otimes \cdots \otimes \sigma_{\vec{a}_m} \mapsto \sum_{\vec{b}_1,\dots,\vec{b}_n} \left( \int_{[Q_C]^{\text{vir}}} \bar{\sigma}_{\underline{\vec{a}}} \cup \bar{\sigma}_{\underline{\vec{b}}} \cup \bar{\sigma}_{1^r}^{r(g+n)+s} \right) \, \sigma_{\vec{b}_1^c} \otimes \cdots \otimes \sigma_{\vec{b}_n^c},
\]
where $Q_{e,C} = \Quot\left((r,-e),\OO_C^{r+s}\right)$ and $Q_C = \bigsqcup_e Q_{e,C}$. 
Recall the identity at the end of Example \ref{ex:special_properties} and the fact that integrals on $Q_{\PP^1}$ define the Gromov-Witten numbers. Thus
\[
	F(0)_2^1 \colon \sigma_{\vec{a}} \otimes \sigma_{\vec{b}} \mapsto \sum_{\vec{c}} \sum_e \langle W_{\vec{a}},W_{\vec{b}},W_{\vec{c}} \rangle_e \, \sigma_{\vec{c}^c} = \sigma_{\vec{a}} * \sigma_{\vec{b}}
\]
is just the quantum product, while
\[
	F(0)_2^0 \colon \sigma_{\vec{a}} \otimes \sigma_{\vec{b}} \mapsto \int_{Q_{\PP^1}} \bar{\sigma}_{\vec{a}} \cup \bar{\sigma}_{\vec{b}} \cup \bar{\sigma}_{1^r}^s = \langle \sigma_{\vec{a}} * \sigma_{\vec{b}},\sigma_{s^r} \rangle
\]
is a pairing that picks out the coefficient of 1 in the quantum product, which is different from the Poincar\'{e} pairing. Witten's TQFT is simply the quantum cohomology ring $QH^*(G)$ equipped with this new pairing
\[
	\langle \sigma_{\vec{a}},\sigma_{\vec{b}} \rangle_W := \langle \sigma_{\vec{a}}*\sigma_{\vec{b}}, \sigma_{s^r} \rangle.
\]

\begin{prop}[Concrete description of Witten's TQFT]{\ } Define the basis $\{ \, \sigma_I := \sum_{\vec{a}} \overline{\sigma_{\vec{a}}(\zeta^I)} \sigma_{\vec{a}} \, \}$ and the constants $a_I := (r+s)^r/\mathrm{Vand}(\zeta^I)$ as above.
\begin{enumerate}[(a)]
\item A semisimple basis is $\{ \, \tilde{e}_I := a_I^{-1/2} \sigma_I \, \}$, with semisimple values $\{ \, \tilde{\lambda}_I = a_I^{1/2} \, \}$.

\item The genus-addition operator is $* \sum_I \sigma_I$ and its $e_I$-eigenvalue is $\tilde{\lambda}_I^2 = a_I$.
\end{enumerate}
\end{prop}

\begin{proof} The computation is similar to the proof of Proposition \ref{p:qctqft}. The key difference is that
\[
	\| \sigma_I \|_W^2 = \langle \sigma_I^2, \sigma_{s^r} \rangle = a_I \langle \sigma_I, \sigma_{s^r} \rangle = a_I
\] 
since the coefficient of 1 in each $\sigma_I$ is 1.
\end{proof}

\begin{rmk}\label{r:gao_comparison} Note that applying $*\sigma_{s^r}$ to the genus-addition operator in Witten's TQFT yields the genus-addition operator for quantum cohomology. Conversely, we can produce Witten's operator from the quantum cohomology operator by applying $*\sigma_{1^r}^r$.
\end{rmk}

Witten's TQFT has the remarkable property that the closed invariants are the Verlinde numbers $V_g^{r,s}$:

\begin{prop} In Witten's TQFT, $F(g)_0^0 = V_g^{r,s}$.
\end{prop}

\begin{proof} We compute $F(g)_0^0$ as the trace of the $(g-1)$th power of the genus-addition operator. Since the eigenvalues of the genus-addition operator are $a_I$, it suffices to prove that $a_I = \prod_{(j,k) \in I \times \bar{I}} |2 \sin \pi \tfrac{j-k}{r+s}|$ because then the claim follows by taking sums of $(g-1)$th powers. We can rewrite this equation as
\begin{equation}\label{eq:witten_verlinde}
	(r+s)^r = \mathrm{Vand}(\zeta^I) \prod_{(j,k) \in I \times \bar{I}} |2 \sin \pi \tfrac{j-k}{r+s}|.
\end{equation}
The key identity is
\[
	r+s = \prod_{n=1}^{r+s-1} 2 \sin \pi \tfrac{n}{r+s},
\]
which can be proved by rewriting $2 \sin \pi \tfrac{n}{r+s} = i(\zeta^{-n/2} - \zeta^{n/2}) = i\zeta^{-n/2}(1-\zeta^n)$ and using the fact that the $\zeta^n$ are all the roots of the polynomial $p(x) = 1 + x + \cdots + x^{r+s-1}$, hence $\prod_n (1-\zeta^n) = p(1) = r+s$.

Now, we use a similar trick to write the Vandermonde norm as
\[
	\mathrm{Vand}(\zeta^I) = \prod_{j,k \in I,\, j \ne k} |2 \sin \pi \tfrac{j-k}{r+s}|
\]
and count how many times $|2 \sin \pi \tfrac{n}{r+s}|$ occurs for each $n$. On the left side of (\ref{eq:witten_verlinde}), each $|2 \sin \pi \tfrac{n}{r+s}|$ has multiplicity $r$. On the right side, for each $j \in I$ and each $n$, there is a unique solution to $j - k \equiv n \pmod{r+s}$ chosen from $\{ 1,\dots,r+s \}$; the corresponding factor is supplied by the Vandermonde norm when the solution $k$ is in $I$. Since each $j \in I$ yields a unique solution in either $I$ or $\bar{I}$, each $|2 \sin \pi \tfrac{n}{r+s}|$ again has multiplicity $r$.
\end{proof}

\subsection{The weighted TQFT}\label{ss:wtqft}

Recall that the weighted TQFT is defined by the maps
\[
	F(g|d)_m^n \colon \sigma_{\vec{a}_1} \otimes \cdots \otimes \sigma_{\vec{a}_m} \mapsto \sum_{\vec{b}_1,\dots,\vec{b}_n} \left( \int_{Q_V} \bar{\sigma}_{\underline{\vec{a}}} \cup \bar{\sigma}_{\underline{\vec{b}}} \right) \, \sigma_{\vec{b}_1^c} \otimes \cdots \otimes \sigma_{\vec{b}_n^c},
\]
where $C$ is a curve of genus $g$, $V$ is a general stable vector bundle on $C$ of rank $r+s$ and degree $d$, and $Q_V = \bigsqcup_{e \in \ZZ} Q_{e,V}$. We prove in \S\ref{s:quot_general} that (after twisting by a line bundle) $V$ can be constructed as a general elementary modification of a trivial bundle, and use this to show that $Q_{e,V}$ has the expected dimension for all $e$. 

Since stable (namely, balanced) bundles of degree 0 on $\PP^1$ are trivial, the maps $F(0|0)_m^n$ are determined by the Gromov-Witten numbers, and in particular $F(0|0)_2^1$ is the quantum product $*$ and $F(0|0)_2^0$ is the Poincar\'{e} pairing $\langle \, , \, \rangle$. Thus the $d=0$ slice of the weighted TQFT is exactly the quantum cohomology TQFT.

We now compute the degree-lowering operator $F(0|{-}1)_1^1$. The balanced bundle of degree $-1$ can be obtained as the kernel $V$ of a general elementary modification $0 \to V \to \OO_{\PP^1}^{r+s} \to \CC_q \to 0$. Moreover, the image of the embedding $\iota \colon Q_{e,V} \hookrightarrow Q_{e,\PP^1}$ is exactly $\overline{W}_{1^r}(q)$ and the pullback under $\iota$ of a Schubert variety $\overline{W}_{\vec{a}}$ is again of type $\overline{W}_{\vec{a}}$ on $Q_{e,V}$ (again, see \S\ref{s:quot_general}). Thus
\[
	\int_{Q_V} \bar{\sigma}_{\underline{\vec{a}}} = \int_{Q_{\PP^1}} \bar{\sigma}_{\underline{\vec{a}}} \cup \bar{\sigma}_{1^r},
\]
which allows us to compute
\[
	F(0|{-}1)_1^1 \colon \sigma_{\vec{a}} \mapsto \sum_{\vec{b}} \left( \int_{Q_V} \bar{\sigma}_{\vec{a}} \cup \bar{\sigma}_{\vec{b}} \right) \sigma_{\vec{b}^c} = \sum_{\vec{b}} \left( \int_{Q_{\PP^1}} \bar{\sigma}_{\vec{a}} \cup \bar{\sigma}_{\vec{b}} \cup \bar{\sigma}_{1^r} \right) \sigma_{\vec{b}^c} = \sigma_{\vec{a}} * \sigma_{1^r},
\]
namely the degree-lowering operator is $*\sigma_{1^r}$. The properties in Example \ref{ex:special_properties} imply that the inverse operator $F(0|1)_1^1$ is $*\sigma_s$ and that $F(g|d)_m^n = F(g|d+r+s)_m^n$.

We can now see how the weighted TQFT contains Witten's TQFT. It follows from $F(0)_1^1 = \mathrm{id} = F(0|0)_1^1$ and $F(0)_2^1 = * = F(0|0)_2^1$ that $F(0)_m^1 = F(0|0)_m^1$ for all $m \ge 1$. The shift in degrees occurs because of the pairing. Since $\langle \sigma_{\vec{a}},\sigma_{\vec{b}} \rangle_W = \langle \sigma_{\vec{a}} * \sigma_{\vec{b}},\sigma_{s^r} \rangle = \langle \sigma_{\vec{a}}*\sigma_{\vec{b}}*\sigma_{s^r}, 1 \rangle$, Witten's pairing can be obtained by (partially) precomposing the Poincar\'{e} pairing with $F(0|r)_1^1 = F(0|{-}s)_1^1 = *\sigma_{s^r}$. Thus $F(0)_2^0 = F(0|-s)_2^0$. Since the partial composition of $F(0)_2^0$ and $F(0)_0^2$ produces the cylinder $F(0)_1^1$, which has degree 0 in the weighted TQFT, we obtain $F(0)_0^2 = F(0|s)_0^2$. Moreover, it follows from Remark \ref{r:gao_comparison} that $F(1)_1^1 = F(1|s)_1^1$. Now, to compute the degree of $F(g)_m^n$ as a map in the weighted TQFT, we write $F(g)_m^n$ as a partial composition of $F(0)_{m+n-1}^1$ with $g$ copies of $F(1)_1^1$ (to increase the genus) and $n-1$ copies of $F(0)_0^2$ (to move a copy of the vector space from the domain to the codomain), and add up the degrees of these maps. We already observed that . Thus each of the $g+n-1$ maps in the composition increases the degree by $s$, so
\[
	F(g)_m^n = F(g|s(g{-}1{+}n))_m^n
\]
and in particular
\[
	V_g^{r,s} = F(g)_0^0 = F(g|s(g{-}1))_0^0.
\]
As mentioned in \S\ref{ss:wtqft}, this last fact says that when $V$ is general of degree $s(g-1)$, the integral of the fundamental class on $Q_V$, which counts the points of the 0-dimensional Quot scheme $Q_{0,V}$, yields the Verlinde number $V_g^{r,s}$.

\subsection{Enumerating all finite Quot schemes}
The weighted TQFT recovers the formula of Holla (\cite{Hol04}) that enumerates the points of all finite Quot schemes of general bundles on curves, including cases that are not computable within Witten's TQFT. For a general bundle $V$ of degree $d$, the Quot scheme $Q_{e,V}$ is finite if and only if $rd + (r+s)e - (g-1)rs = 0$. Setting $a = \gcd(r,s)$, $r' = r/a$, and $s' = s/a$, the full solution set is
\[
	\{ \, e = r' \gamma, d = (g-1)s - (r'+s')\gamma \mid \gamma \in \ZZ \, \};
\]
imposing the equivalence relation of twisting both the subsheaves and $V$ by line bundles, we can think of $\gamma \in \ZZ/a\ZZ$.
The finite Quot schemes whose points are counted by the Verlinde numbers come from the solution $0 \in \ZZ/a\ZZ$. For $r$ and $s$ relatively prime, these are all the solutions, but when $a > 1$ there are others. A natural choice of representatives of the solution classes is $e = r'\gamma$ and $d = -(r'+s')\gamma+s(g-1)$ for $0 \le \gamma < a$. We compute $F\left(g|{-}(r'+s')\gamma+s(g{-}1)\right)$ by taking the trace of the composition of the $(g-1)$th power of Witten's genus-addition operator and the $(r'+s')\gamma$th power of the degree-lowering operator. Since the $e_I$-eigenvalue of the degree-lowering operator is $\sigma_{1^r}(\zeta^I)$, we obtain the formula
\[
	\# Q_{r'\gamma,V} = (r+s)^{r(g-1)} \sum_I \left( \frac{\sigma_{1^r}(\zeta^I)^{(r'+s')\gamma}}{\mathrm{Vand}(\zeta^I)}\right)^{g-1},
\]
thus proving the Main Corollary in \S \ref{ss:synthesis}.

\subsection{More geometric interpretation of the weighted TQFT}

Because the $Q_{e,V}$ all have the expected dimension, we can interpret the maps $F(g|d)_0^N$ geometrically as follows. Given a curve $C$ of genus $g$, a general vector bundle $V$ on $C$ of degree $d$, and distinct points $p_1,\dots,p_{N}$ on $C$, there is a rational map
\[
	\ev_{p_1,\dots,p_N}^{V} \colon \bigsqcup_{e \in \ZZ} Q_{e,V} \dashrightarrow G^{N}
\]
defined by restricting subsheaves $E \subset V$ to their fibers at each $p_i$. Let
\[
	\eta_{g,d,N} \in H^*(G,\CC)^{\otimes N}
\]
denote the sum of the pushforwards of the fundamental classes of the $Q_{e,V}$ (which do not depend on the choices of curve, general vector bundle in moduli, or points $p_i$).

\begin{rmk} The data encoded by the weighted TQFT are equivalent to these $\eta_{g,d,N}$. To get $F(g|d)_0^N$, simply define the linear map $\CC \to H^*(G,\CC)^{\otimes N}$ by $1 \mapsto \eta_{g,d,N}$. Choosing a partition $N = m + n$, dualizing the first $m$ copies of $H^*(G,\CC)$ yields a linear map $(H^*(G,\CC)^*)^{\otimes m} \to H^*(G,\CC)^{\otimes n}$, and Poincar\'{e} duality yields an isomorphism between $H^*(G,\CC)$ and its dual, hence the class $\eta_{g,d,m+n}$ determines a linear map
\[
	F(g|d)_m^n \colon H^*(G,\CC)^{\otimes m} \to H^*(G,\CC)^{\otimes n}
\]
that coincides with the map defined in the weighted TQFT.
\end{rmk}

The dimension of $Q_{e,V}$ is $rd+(r+s)e - (g-1)rs$, so its image in $G^N$ is a class of degree $(N+g-1)rs - rd - (r+s)e$. It can be convenient to distinguish the image classes of $Q_{e,V}$ in $G^N$ by multiplying them by $q^e$, where $q$ is a formal variable. Viewing $q$ as having degree $r+s$ makes $\eta_{g,d,N}$ into a homogeneous class of degree $(N+g-1)rs - rd$.

We can use the weighted TQFT relations (with the formal variable $q$) to describe how these image classes change as we modify $C$, $V$, and $N$. Let $\pi_k \colon G^N \to G$ denote projection to the $k$th factor and $\hat{\pi}_k \colon G^{N} \to G^{N-1}$ denote projection to all but the $k$th factor. Then for any $1 \le k \le N$, the following are some identities implied by the weighted TQFT:
\begin{enumerate}
\item Genus-addition: $\eta_{g+1,d,N} = \eta_{g,d,N} * \pi_k^* \left( \sum_{\vec{a}} \sigma_{\vec{a}} * \sigma_{\vec{a}^c} \right)$;
\item Degree-lowering: $\eta_{g,d-1,N} = \eta_{g,d,N} * \pi_k^* (\sigma_{1^r})$;
\item Degree-raising: $\eta_{g,d+1,N} = \eta_{g,d,N} * \pi_k^* (\sigma_s * q^{-1})$;
\item Forgetful: $\eta_{g,d,N-1} = (\hat{\pi}_k)_* \eta_{g,d,N}$.
\end{enumerate}
Here $*\sum_{\vec{a}} \sigma_{\vec{a}} * \sigma_{\vec{a}^c}$ is another formula for the genus-addition operator. The classes $\sigma_{1^r}$ and $\sigma_s*q^{-1}$ are multiplicative inverses in quantum cohomology. The degrees of the classes $\sum_{\vec{a}} \sigma_{\vec{a}} * \sigma_{\vec{a}^c}$, $\sigma_{1^r}$, and $\sigma_s * q^{-1}$ are $rs$, $r$, and $-r$, respectively, which matches the effect of the corresponding operations on the dimension of the Quot scheme. Forgetting one of the marked points corresponds to capping one of the boundary circles in the weighted TQFT.

In the particular case when $Q_{e,V}$ is finite and reduced for some $e$, its image class will count its points. The power of $q$ can be used to identify which $e$ yields the finite Quot scheme:
\[
	\int_{G^N} \eta_{g,d,N} = q^e \, \# Q_{e,V}.
\]
We can obtain formulas for $\# Q_{e,V}$ using the above identities as we did in the weighted TQFT.

\section{Background on Quot schemes}\label{s:background2}
Most of the material in this section is a straightforward generalization of results in \cite{Ber94} and \cite{Ber97} to Quot schemes of arbitrary vector bundles.

\subsection{Schubert varieties on Quot schemes}

We define Schubert varieties carefully and generally for arbitrary Quot schemes. Let $V$ be a vector bundle of rank $r+s$ on a curve $C$, let $Q_{e,V} = \Quot\big((r,-e),V\big)$, and let $0 \to \mathcal{E} \to \pi^*V \to \mathcal{F} \to 0$ denote the universal sequence on $C \times Q_{e,V}$. Choose a point $p \in C$, a full flag $V_{\bullet}$ of the fiber $V(p)$, and a partition $\vec{a}$. The $V_i \subset V(p)$ have cokernels $V(p) \surj V^i$.
Restricting the universal sequence to $\{ p \} \times Q_{e,V} \cong Q_{e,V}$ yields an exact sequence
\[
	0 \to \mathcal{E}|_{\{p\} \times Q_{e,V}} \to V(p) \otimes \OO \to \mathcal{F}|_{\{p\} \times Q_{e,V}} \to 0
\]
on the Quot scheme.
\begin{defn} The \emph{Schubert variety} $\overline{W}_{\vec{a}}(p,V_{\bullet}) \subset Q_{e,V}$ is the intersection for all $1 \le i \le r$ of the degeneracy loci where the compositions
\[
	\mathcal{E}|_{\{p\} \times Q_{e,V}} \to V(p) \otimes \OO \surj V^{s+i-a_i} \otimes \OO
\]
have kernel of dimension $\ge i$. 
\end{defn}

In the case when $\overline{W}_{\vec{a}}(p,V_\bullet)$ has pure codimension $|\vec{a}|$ for all $p$ and $V_\bullet$, the following lemma guarantees that the Schubert cycle $\bar{\sigma}_{\vec{a}} = [\overline{W}_{\vec{a}}(p,V_{\bullet})]$ in the cohomology of the Quot scheme is independent of $p$ and $V_\bullet$. 
\begin{lem} Suppose $\overline{W}_{\vec{a}}(p,V_\bullet)$ has pure codimension $|\vec{a}|$ for all $p \in C$ and flags $V_\bullet$. Then the cohomology class $\bar{\sigma}_{\vec{a}}$ of $\overline{W}_{\vec{a}}(p,V_\bullet)$ in $Q_{e,V}$ is independent of $p$ and $V_{\bullet}$.
\end{lem}


\begin{proof}  Since the $\overline{W}_{\vec{a}}(p,V_\bullet)$ have pure codimension $|\vec{a}|$, we can apply a theorem of Kempf-Laksov (\cite{KL74}) to the maps $\mathcal{E}|_{\{p\} \times Q_{e,V}} \to V(p) \otimes \OO \to V^i \otimes \OO$ to express the cohomology class of the degeneracy locus as a determinantal formula involving only the Chern classes of $\mathcal{E}|_{\{p\} \times Q_{e,V}}$ (and of the flag, but these are trivial). But since $\mathcal{E}$ is a vector bundle over $C \times Q_{e,V}$, the Chern classes of the restrictions of $\mathcal{E}$ over $p \in C$ are independent of $p$.
\end{proof}


We need Schubert varieties and their intersections to have the right codimension if we hope to get a useful intersection theory on the Quot scheme. Indeed, if intersections on $Q_{e,V}$ are well-behaved, then it makes sense to define analogs of the Gromov-Witten numbers as
\[
	\int_{Q_{e,V}} \bar{\sigma}_{\vec{a}_1} \cup \cdots \cup \bar{\sigma}_{\vec{a}_N}
\]
for $|\vec{a}_1| + \cdots + |\vec{a}_N| = \dim Q_{e,V}$.

\begin{defn} Suppose $Q_{e,V}$ is equidimensional. An intersection $W$ of Schubert varieties on $Q_{e,V}$ is \emph{proper} if it is empty or of pure codimension equal to the total size $A$ of the partitions. The intersection $W$ has \emph{failure} $\nu$ if its codimension is $A-\nu$ (we use similar terminology for the preimage of $W$ under a morphism to $Q_{e,V}$). We call $W$ a \emph{top intersection} if $A=\dim Q_{e,V}$.
\end{defn}

On the open (possibly empty) subscheme $U_{e,V} \subset Q_{e,V}$ parametrizing quotients that are torsion-free (which generalizes $\Mor_e(\PP^1,G) \subset Q_{e,\PP^1}$), there is no difficulty in proving that Schubert varieties intersect properly. For each $p \in C$, there is a morphism
\[
	\ev_p \colon U_{e,V} \to \Gr\big(r,V(p)\big), \quad [E \subset V] \mapsto [E(p) \subset V(p)]
\]
induced by $\mathcal{E}|_{\{p\} \times U_{e,V}} \inj V(p) \otimes \OO$ on $\{p\} \times U_{e,V} \simeq U_{e,V}$. The presence of morphisms to the Grassmannian makes it easy to control the behavior of intersections of Schubert varieties. Given distinct points $p_1,\dots,p_N$, the following lemma guarantees that the intersection $\ev_{p_1}^{-1}(W_{\vec{a}_1}) \cap \cdots \cap \ev_{p_N}^{-1}(W_{\vec{a}_N})$ is proper on each component of $U_{e,V}$.

\begin{lem}\label{l:morphism_proper} Let $X$ be an irreducible scheme with morphisms $f_i \colon X \to G$ for all $1 \le i \le N$. Then for all choices of partitions $\vec{a}_1,\dots,\vec{a}_N$ and Schubert varieties $W_{\vec{a}_1},\dots,W_{\vec{a}_N}$ defined using general flags, the preimage $f_1^{-1}(W_{\vec{a}_1}) \cap \cdots \cap f_N^{-1}(W_{\vec{a}_N})$ is empty or has pure codimension $A = |\vec{a}_1| + \cdots + |\vec{a}_N|$ in $X$.
\end{lem}

\begin{proof} Induction on $N$. The base case $N=1$ is Kleiman's theorem (\cite{Kle74}). For the inductive step, assume $X' = f_1^{-1}(W_{\vec{a}_1}) \cap \cdots \cap f_{N-1}^{-1}(W_{\vec{a}_{N-1}})$ has pure codimension $A-|\vec{a}_N|$ in $X$. Restricting $f_N$ to each component $Y$ of $X'$, the base case guarantees that $f_N|_Y^{-1}(W_{\vec{a}_N})$ is empty or has pure codimension $|\vec{a}_N|$ in $Y$, which completes the proof.
\end{proof}

The challenge in proving that Schubert varieties intersect properly is to understand how they intersect the boundary $U_{e,V}^c$ of the Quot scheme $Q_{e,V}$. In general, the Schubert varieties will not intersect properly on the boundary, but it suffices to prove that the failure on the boundary is no greater than the codimension of the boundary. In the next section, we describe a strategy for proving properness that involves stratifying $U_{e,V}^c$ and controlling the failure on each stratum. This motivates a technical study of the recursive structure of $U_{e,V}^c$.

\subsection{Boundary of the Quot scheme}\label{ss:boundary}

To motivate the constructions in this section, we begin with an outline of a proof of properness of Schubert intersections. We fix the vector bundle $V$ on the curve $C$ and drop it from the notation: let $Q_e := \Quot\big((r,-e),V\big)$, let $U_e$ denote the (possibly empty) subscheme in $Q_e$ where the quotient is torsion-free, and let $\mathcal{E}_{e}$ denote the universal subsheaf, which is a vector bundle on $C \times Q_e$.

Let $W$ denote an intersection of general Schubert varieties. The idea for proving $W$ is proper in $Q_e$ is to control the codimension of $W$ on a stratification of $Q_e$. One has to prove that the failure of $W$ to be proper on each stratum is no greater than the codimension of that stratum in $Q_e$.

\begin{strategy}[proving general Schubert intersections $W$ are proper on $Q_e$]\label{strategy:proper}{\ }
\begin{enumerate}[(1)]
\item On $U_e$, we use Lemma \ref{l:morphism_proper}.
\item On the boundary $U_e^c$, which is the image of a Grassmann bundle $\pi \colon \Gr(\mathcal{E}_{e-1},1) \to C \times Q_{e-1}$ under a map $\beta_1 \colon \Gr(\mathcal{E}_{e-1},1) \to Q_e$ (Proposition \ref{p:structure}), it suffices to control the codimension of $\beta_1^{-1}(W)$.
\begin{enumerate}
\item $\beta_1^{-1}(W)$ splits into two types of components (Corollary \ref{c:types}).
\item The Type 1 component comes from the base $Q_{e-1}$ and is handled by induction on $e$.
\item Each Type 2 component, which is contained in $G_{1,p} := \pi^{-1}(\{p\} \times Q_{e-1})$ for some $p$, consists of an intersection of Schubert varieties in $Q_{e-1}$ (handled by induction on $e$) with a degeneracy locus $\hat{W}_{\vec{b}}(p)$.
\item By stratifying $G_{1,p}$ as $U_{1,p} \sqcup \beta^{-1}(\beta_{\ell,p}(U_{\ell,p}))$ (Corollary \ref{c:stratification}) and identifying the fibers of $\beta$ over each $\beta_{\ell,p}(U_{\ell,p})$, we reduce to understanding $\hat{W}_{\vec{b}}(p)$ on $U_{\ell,p}$.
\item On $U_{\ell,p}$, $\hat{W}_{\vec{b}}$ is the preimage of $W_{\vec{b}}$ under a map $U_{\ell,p} \to \Gr(r-\ell,\CC^{r+s})$ (Remark \ref{r:degeneracy_locus}), so we use Lemma \ref{l:morphism_proper}.
\end{enumerate}
\end{enumerate}
\end{strategy}

\begin{rmk} For this stratification to give the right codimension estimates, we need to know that the Quot schemes $Q_{e-\ell}$ have the expected dimension for all $0 \le \ell \le r$. When the vector bundle is trivial, this is true on $\PP^1$ for all $e$, and the strategy can be used to prove properness (Bertram's proof of Theorem \ref{t:proper_P1}). When the vector bundle is trivial but the curve has positive genus, this is only true for $e \gg 0$, but the strategy can still be used if some care is taken with the induction on $e$ (Lemma \ref{l:proper_trivial}). For general vector bundles, we use a different approach.
\end{rmk}

We now proceed to fill in the details of the constructions. Given $p \in C$ and an integer $1 \le \ell \le r$, the fiber of the Grassmann bundle $\Gr(\mathcal{E}_e,\ell)$ over a point $(p,x=[0 \to E \to V \to F \to 0]) \in C \times Q_e$ parametrizes quotients $\mathcal{E}_e|_{(p,x)} = E(p) \twoheadrightarrow \CC^\ell$. By composing $E \surj E(p) \surj \CC^\ell$, these quotients induce elementary modifications $0 \to E' \to E \to \CC_p^\ell \to 0$ in which $\deg E' = \deg E - \ell$. Since $E' \hookrightarrow E \hookrightarrow V$, there is a short exact sequence $0 \to E' \to V \to F' \to 0$, and the assignment  $[E(p) \twoheadrightarrow \CC^\ell] \mapsto [0 \to E' \to V \to F' \to 0]$ defines a set-theoretic map $\beta_\ell$:
\[\xymatrix{
	\Gr(\mathcal{E}_e,\ell) \ar[d]_{\pi_\ell} \ar[r]^-{\beta_\ell} & Q_{e+\ell} \\
	C \times Q_e
}\]
Let $0 \to \mathcal{S} \to \pi_\ell^* \mathcal{E}_e \to \mathcal{Q} \to 0$ denote the tautological sequence of vector bundles on $\Gr(\mathcal{E}_e,\ell)$, which is equipped with a map $\pi_\ell^* \mathcal{E}_e \to \pi_C^* V$ from the tautological sequence on $Q_e$.


\begin{prop}[\cite{Ber97}]\label{p:structure} In the setting above,
\begin{enumerate}[(a)] 
\item $\beta_\ell$ is a morphism of schemes.
\item $\mathrm{im}(\beta_\ell)$ is closed in $Q_{e+\ell}$ and contains exactly those points $x \in Q_{e+\ell}$ such that the universal quotient $\mathcal{F}_{e+\ell}$ has rank $\ge s+\ell$ at $(p,x) \in C \times Q_{e+\ell}$ for some $p \in C$.
\item The restriction of $\beta_\ell$ to $\pi_\ell^{-1}(C \times U_e)$ is an embedding.
\item Let $p \in C$ and let $\overline{W}_{\vec{a}}(p)$ denote a Schubert variety in $Q_{e+\ell}$ defined using a full flag $V_{\bullet}$ in $V(p)$. Then
\[
	\beta_\ell^{-1}\big(\overline{W}_{\vec{a}}(p)\big) = \pi_\ell^{-1}\big(C \times \overline{W}_{\vec{a}}(p)\big) \cup \hat{W}_{a_{\ell+1},\dots,a_r}(p),
\]
where $\hat{W}_{b_1,\dots,b_{r-\ell}}(p)$ is the degeneracy locus inside $\pi_\ell^{-1}(\{p\} \times Q_e)$ where the kernel of $\mathcal{S} \to V^{s+\ell+j-b_j} \otimes \OO$ has rank $\ge j$ for all $1 \le j \le r-\ell$.
\end{enumerate}
\end{prop}

\begin{rmk}\label{r:degeneracy_locus} One can think of $\hat{W}_{\vec{b}}(p) \subset \pi_\ell^{-1}(\{p\} \times Q_e)$, where $\vec{b}$ has length $r-\ell$, as follows. On the open locus $\pi_\ell^{-1}(\{p\} \times U_e)$, the elementary modifications yield $0 \to E' \to V \to F' \to 0$, where $E'(p) \to V(p)$ has rank exactly $r-\ell$. Thus we get a map $\pi_\ell^{-1}(\{p\} \times U_e) \to \Gr\big(r-\ell,V(p)\big)$, and pulling back the Schubert variety $W_{\vec{b}}(V_{\bullet})$ and taking its closure yields $\hat{W}_{\vec{b}}(p)$.
\end{rmk}

The fact that the varieties $\hat{W}_{\vec{b}}(p)$ based at different points $p$ are disjoint proves the following corollary.

\begin{cor}\label{c:types} In the setting above, let $\overline{W}_{\vec{a}_1}(p_1),\dots,\overline{W}_{\vec{a}_N}(p_N)$ be Schubert varieties in $Q_{e+\ell}$ defined at distinct points $p_i$, and let $W$ denote their intersection. Then, up to reindexing the $\vec{a}_i$, $\beta_\ell^{-1}(W)$ is a union of intersections of the following types:
\begin{enumerate}[Type 1:]
\item $\pi_\ell^{-1}\left(C \times \big(\overline{W}_{\vec{a}_1}(p_1) \cap \cdots \cap \overline{W}_{\vec{a}_N}(p_N)\big) \right)$;
\item $\pi_\ell^{-1}\left(p_N \times \big(\overline{W}_{\vec{a}_1}(p_1) \cap \cdots \cap \overline{W}_{\vec{a}_{N-1}}(p_{N-1})\big)\right) \cap \hat{W}_{(\vec{a}_N)_{\ell+1},\dots(\vec{a}_N)_r}$.
\end{enumerate}
\end{cor}

The proposition is stated for $C = \PP^1$ and $V$ the trivial vector bundle in \cite{Ber97}. The proof given there generalizes without requiring modification. For details, see \cite{Gol17}.

In order to handle the degeneracy loci $\hat{W}_{\vec{b}}(p)$, we need to stratify the fiber of the Grassmann bundle over $p$. Let $G_{\ell,p}$ denote the locus $\pi_\ell^{-1}(\{p\} \times Q_{e-\ell}) \subset \Gr(\mathcal{E}_{e-\ell},\ell)$ parametrizing elementary modifications of the universal subsheaf at the fixed point $p$. As before, there are maps
\[\xymatrix{
	G_{\ell,p} \ar[d]_{\pi_{\ell,p}} \ar[r]^-{\beta_{\ell,p}} & Q_e \\
	Q_{e-\ell}
}\]
and there is a universal sequence $0 \to \mathcal{S}_{\ell,p} \to \pi_{\ell,p}^* \mathcal{E}_{e-\ell}|_{\{p\} \times Q_{e-\ell}} \to \mathcal{Q}_{\ell,p} \to 0$ on $G_{\ell,p}$. Let $U_{\ell,p} \subset G_{\ell,p}$ denote the open subscheme where the fibers of the map $\mathcal{S}_{\ell,p} \to V(p) \otimes \OO$ are injective. Then $\beta_{\ell,p}$ is injective on $U_{\ell,p}$ and the image $Z_{\ell,p} := \beta_{\ell,p}(U_{\ell,p})$ is the locally-closed locus in $Q_e$ where the inclusion $E \subset V$ has rank exactly $r-\ell$ at $p$. Let $U_{\ell,p}^c$ denote the complement of $U_{\ell,p}$ in $G_{\ell,p}$.

\begin{cor}\label{c:stratification}
There is a set-theoretic stratification
\[
	U_{\ell,p}^c = \bigsqcup_{\ell < \ell' \le r} \beta_{\ell,p}^{-1}(Z_{\ell',p}) \subset G_{\ell,p}.
\]
Moreover, in the case when the Quot schemes $Q_{e-\ell}$ are of the expected dimension for all $0 \le \ell \le r$, then
\begin{equation}\label{eq:quot_strat_dim}
	\dim Q_e - \dim Z_{\ell,p} = \ell(s+\ell),
\end{equation}
\begin{equation}\label{eq:grassmann_strat_dim}
	\dim G_{\ell,p} - \dim \beta_{\ell,p}^{-1}(Z_{\ell',p})
	= (\ell' - \ell)(s + \ell').
\end{equation}
\end{cor}

\begin{proof}
Since the $Z_{\ell,p}$ stratify $Q_e$, their preimages stratify $U_{\ell,p}^c$. (1) follows from the dimension formula for Quot schemes and Grassmannians. For (2), we need to identify the fiber. At a point $x=[0 \to E \to V \to F \to 0]$ in $Z_{\ell',p}$, the torsion subsheaf $T$ of $F$ has fiber of dimension exactly $\ell'$ at $p$, hence $T$ is a direct sum of $\ell'$ quotients of $\OO_C$ and has a canonical subsheaf $\CC_p^{\ell'} \subset T \subset F$ such that every map $\CC_p \to F$ factors through $\CC_p^{\ell'}$. The fiber of $\beta_{\ell,p}$ over $x$ consists of all $[0 \to E' \to V \to F' \to 0]$ such that there is an elementary modification $E' \surj \CC_p^\ell$ yielding $E$ as the kernel. By the snake lemma, such modifications correspond to maps $\CC_p^\ell \inj F$ (whose cokernel is $F'$). Thus the fiber over $x$ is isomorphic to $\Gr(\ell,\CC^{\ell'})$, which has dimension $\ell(\ell'-\ell)$. Combining this with the dimension formula for Quot schemes gives the result.
\end{proof}

\section{Quot schemes of trivial bundles}\label{s:quot_trivial}
Quot schemes of trivial bundles are known to be irreducible of the expected dimension when the degree of the subsheaf is sufficiently negative. We prove that in this case, intersections of general Schubert varieties are proper. Since we consider mainly Quot schemes of trivial bundles in this section, we use the notation $Q_{e,C} = \Quot\big( (r,-e), \OO_C^{r+s} \big)$. When the curve is clear from context, we may even write just $Q_e$. As usual, for general $V$ we write $Q_{e,V}=\Quot\big( (r,-e),V \big)$ and let $U_{e,V}$ denote the open subscheme of torsion-free quotients.

\subsection{Quot schemes of trivial bundles on $\PP^1$}\label{ss:quot_triv_P1} It was known to Grothendieck that for all $e$, the Quot scheme $Q_{e,\PP^1}$ is irreducible, smooth, of the expected dimension $(r+s)e + rs$ (or empty if this number is negative), and contains $\Mor_e(\PP^1,G)$ as a dense open subscheme. There is a theorem of Bertram ensuring that the intersection theory on the Quot scheme is as nice as possible.
\begin{thm}[\cite{Ber97}]\label{t:proper_P1} For all $e \in \ZZ$,
\begin{enumerate}[(a)]
\item Intersections $W$ of general Schubert varieties in $Q_{e,\PP^1}$ are proper;
\item $W \cap \Mor_e(\PP^1,G)$ is dense in $W$;
\item top intersections $W$ are finite, reduced, and contained in $\Mor_e(\PP^1,G)$.
\end{enumerate}
\end{thm}
Bertram's proof is similar to the outline given in Strategy \ref{strategy:proper}. We make a more general argument in the next subsection.

\subsection{Quot schemes of trivial bundles on curves of any genus}

As the following example illustrates, the properness and denseness claims in Bertram's theorem can fail for curves of positive genus when $e$ is small but positive.
\begin{ex} Let $C$ be a curve of genus $1$. As usual $Q_0 = G$, which has dimension $rs$ rather than the expected dimension 0, but Schubert varieties do intersect properly in $G$. A simple example where properness fails is $Q_1$ with $r=1$. Degree-one line bundles $L$ on $C$ have only a single section, hence are not globally generated, so the maps $L^* \to \OO_C^{r+s}$ must drop rank at a point. Thus every quotient in $Q_1$ has torsion, hence must be of the form $\OO^s \oplus \CC_p$ for some $p \in C$. Thus the map $\beta$ in the diagram
\[\xymatrix{
	G_{1,p} \ar[d]_{\pi} \ar[r]^{\beta} & Q_1 \\
	C \times Q_0
}\]
is an isomorphism. By chance, $Q_1$ has the expected dimension $s+1$ (since $rs + r = r+s$ when $s=1$). However, the Schubert variety $\overline{W}_s(p)$ splits into two pieces on $\Gr(1)$. The first is $\pi^{-1}(C \times \overline{W}_s(p))$, which has the correct codimension $s$ in $G_{1,p}$, but the second is the entire preimage $\pi^{-1}(\{p\} \times Q_0)$, which has only codimension 1 in $G_{1,p}$.
\end{ex}

Despite the disheartening example, Quot schemes of trivial bundles on curves of arbitrary genus are well-behaved when $e$ is sufficiently large. In fact, this is true for Quot schemes of arbitrary vector bundles:
\begin{thm}[\cite{PR03}]\label{t:quot_large_degree} For any vector bundle $V$ on $C$, there is an integer $e_0$ such that for all $e \ge e_0$, $Q_{e,V}$ is irreducible, generically smooth, of the expected dimension $rd + (r+s)e - rs(g-1)$, and the subscheme $U_{e,V}$ of torsion-free quotients is open and dense.
\end{thm}

\begin{rmk} The case $V = \OO_C^{r+s}$ was proved earlier in \cite{BDW96}. This case is particularly geometric since $Q_{e,C}$ compactifies the space of maps $\Mor_e(C,G)$.
\end{rmk}

With this theorem, we can implement Strategy \ref{strategy:proper} for $Q_{e,V}$ with $V$ arbitrary and $e \gg 0$, with the slight complication that the induction on $e$ requires some care.


\begin{lem}\label{l:proper_trivial} Let $V$ be any vector bundle on $C$. Let $e_0$ be chosen such that $Q_e = Q_{e,V}$ has the properties listed in Theorem \ref{t:quot_large_degree} for all $e \ge e_0$. Let $\nu \ge 0$ be the maximum failure over all Schubert intersections on $Q_{e_0+i}$ for all $0 \le i \le r$. Then for all $t \ge 0$, the maximum failure of general Schubert intersections on $Q_{e_0+r+t}$ is $\le \max(\nu-t,0)$.
\end{lem}

\begin{proof}
Induction on $t$. The base case $t=0$ is trivial by definition of $\nu$. Suppose $t > 0$ and that $W$ is a Schubert intersection on $Q_{e_0+r+t}$ with failure $\nu' > 0$. We show that $\nu' \le \nu - t$ by proving appropriate inequalities on a stratification of $Q_{e_0+r+t}$.

As usual, the Schubert intersections are proper on the (dense) open subscheme of morphisms. The boundary is the image of the map $\beta$ in the diagram
\[\xymatrix{
	\Gr(1):=\Gr(\mathcal{E}_{e_0+r+t-1},1) \ar[d]_{\pi} \ar[r]^-{\beta} & Q_{e_0+r+t} \\
	C \times Q_{e_0+r+t-1}.
}\]
By choice of $e_0$, the image of $\beta$ has codimension $s$ in $Q_{e_0+r+t}$, so it suffices to prove that $\beta^{-1}(W)$ has failure $\le \nu - t + s$ in $\Gr(1)$. Recall from Corollary \ref{c:types} that $\beta^{-1}(W)$ decomposes into two types of components in $\Gr(1)$. For the Type 1 component, the inductive hypothesis guarantees that $W$ has failure at most $\nu - t + 1$ on $Q_{e_0+r+t-1}$, which suffices since $s \ge 1$.
  
Each Type 2 component is the intersection of $\pi^{-1}(C \times W')$, where $W'$ is an intersection of Schubert varieties in $Q_{e_0+r+t-1}$, with a degeneracy locus $\hat{W}_{a_{\ell+1},\dots,a_r}(p)$ contained in $G_{1,p}=\pi^{-1}(\{p\} \times Q_{e_0+r+t-1})$ for some $p$. Since $G_{1,p}$ has codimension one in $\Gr(1)$, it suffices to show that the failure of this component on $G_{1,p}$ is $\le \nu-t+s+1$. By Corollary \ref{c:stratification}, there is a stratification
\[
	G_{1,p} = U_{1,p} \sqcup \bigsqcup_{1 \le \ell \le r} \beta_{p}^{-1}(Z_{\ell,p}),
\]
where $\beta_p$ is the restriction of $\beta$ to $G_{1,p}$ and $Z_{\ell,p}$ is the image in $Q_{e_0+r+t}$ of the open set $U_{\ell,p} \subset G_{\ell,p}$ consisting of elementary modifications yielding a subsheaf that drops rank by exactly $\ell$ at $p$. Moreover, $\beta_p^{-1}(Z_{\ell,p})$ has codimension $(\ell-1)(s+\ell)$ in $G_{1,p}$. Thus it suffices to prove that the failure of this Type 2 component on each $U_{\ell,p}$ is
\[
	\le \nu-t+s+1 + (\ell-1)(s+\ell) = \nu - t + 1 + \ell s + \ell(\ell - 1). \tag{$\dagger$}
\]

There are structure maps $U_{\ell,p} \to Q_{e_0+r+t-\ell}$ and by induction the failure of $W'$ is $\le \nu - \max(t-\ell,0)$ or equals 0. The degeneracy locus $\hat{W}_{a_{\ell+1},\dots,a_r}(p)$ is the preimage of $W_{a_{\ell+1},\dots,a_r}$ under the map $U_{\ell,p} \to \Gr(\ell,\CC^{r+s})$, yielding failure $\sum_{i=1}^{\ell} a_i \le \ell s$. Thus the failure on $U_{\ell,p}$ is
\[
	\le \nu - \max(t-\ell,0) + \ell s
\]
or 0, which together with the inequalities $(\ell-1)^2 \ge 0$ and $\max(t-\ell,0) - (t- \ell) \ge 0$ implies 
the inequality ($\dagger$).
\end{proof}

The lemma allows us to deduce the following generalization of Theorem \ref{t:proper_P1}. The statement about top intersections appeared in \cite{Ber94}.

\begin{prop}\label{p:quot_trivial_props} Let $V$ be arbitrary. Then for all $e \gg 0$,
\begin{enumerate}[(a)]
\item Intersections $W$ of general Schubert varieties in $Q_{e,V}$ are proper;
\item $W \cap U_{e,V}$ is dense in $W$;
\item Top intersections $W$ are finite, reduced, and contained in $U_{e,V}$.
\end{enumerate}
\end{prop}

\begin{proof} The properness claim follows from the lemma by taking $e \ge e_0 + r + \nu$. For the denseness statement, note that since $Q_{e,V}$ is irreducible and the codimension of $W$ cannot be greater than its total size $A$ since it is a degeneracy locus, it suffices to show that $W$ actually has failure $< s$ in $\Gr(1)$. We can achieve this by taking $e \ge e_0 + r + \nu$ and copying the proof of the lemma, replacing the $\nu-t$ appearing in the inductive estimates on the failure with zero. For the statement about top intersections, the only part left to be shown is reducedness. But since $W$ is contained in $U_{e,V}$ and $Q_{e,V}$ is generically reduced, the general fibers of the evaluation map to $G^N$ (where $N$ is the number of Schubert varieties) are finite and hence reduced, so the intersection with general Schubert varieties will avoid the branch locus.
\end{proof}

In the next sections, we will need only the case when $V = \OO_C^{r+s}$. For this case, an important result ensures that the intersection theory of the Quot scheme is independent of the choice of the curve in moduli.
\begin{prop}[\cite{Ber94}]\label{p:chern_classes} For all $e \gg 0$,
\begin{enumerate}[(a)]
\item The intersection numbers on $Q_{e,C}$ can be computed as an integral of Chern classes of the universal subsheaf;
\item The intersection numbers on $Q_{e,C}$ do not depend on the choice of smooth curve $C$ of genus $g$. 
\end{enumerate}
\end{prop}

\begin{proof}
(a): The cohomology classes of the dual of the universal subbundle on the Grassmannian (which are $\sigma_{1^k}$ for $1 \le k \le r$) generate the cohomology ring of the Grassmannian. On $\Mor_e(C,G)$, the same formulas express $\bar{\sigma}_{\vec{a}}$ in terms of the Chern classes of the dual of the universal subbundle on $Q_{e,C}$, and the previous proposition ensures that any difference on the boundary is irrelevant for computing intersection numbers.

(b): This is Proposition 1.5 in \cite{Ber94}. Since we will make a similar argument for Proposition \ref{l:nodal_deformation}, we will just summarize the key steps in the proof. Consider a family of smooth curves of fixed genus over a base curve $B$. The relative Quot scheme $Q \to B$ of the trivial bundle contains the Quot schemes of trivial bundles over each of the curves in the family as its fibers over $B$. The intersection numbers of each Quot scheme can be computed in terms of Chern classes of its universal subbundle, and flatness of $Q$ implies that these products of Chern classes are independent of the base point (by expressing them in terms of Chern classes of the universal subbundle of the relative Quot scheme).
\end{proof}

Part (a) generalizes to arbitrary $V$ without any trouble, but (b) is unclear.

\section{Quot schemes of general vector bundles}\label{s:quot_general}
In the previous section, we showed that intersections of Schubert varieties are proper on Quot schemes of trivial bundles if the subsheaves have sufficiently negative degree. In order to get properness results on Quot schemes of general vector bundles $V$, we embed them in Quot schemes of trivial bundles by using elementary modifications. We conclude that for a very general vector bundle, \emph{all} Quot schemes have the expected dimension and proper Schubert intersections.

\subsection{General elementary modifications produce general vector bundles}\label{ss:elementary_modification}

We begin with some terminology. Let $C$ be a curve of genus $g$. We will say that a vector bundle $V$ on $C$ is \emph{stable} if
\[
	\begin{cases}
		\text{$V$ is balanced} & \text{if $g = 0$}; \\
		\text{$V$ is semistable} & \text{if $g = 1$}; \\
		\text{$V$ is stable} & \text{if $g \ge 2$}.
	\end{cases}
\]
A vector bundle on $\PP^1$ is \emph{balanced} if its splitting $V=\bigoplus_{i=1}^{\mathrm{rk}(V)} \OO_{\PP^1}(d_i)$ has the property that $|d_i - d_j| \le 1$ for all $1 \le i,j \le \mathrm{rk}(V)$. There is a moduli space $M(r,d)$ parametrizing all vector bundles of rank $r$ and degree $d$ that are stable in this sense (on $\PP^1$, $M(r,d)$ is just a point). When we call the vector bundle $\emph{general}$, we mean that it is stable and does not lie on a finite collection of proper closed subschemes of $M(r,d)$ when $g \ge 1$. Even stronger, we say the vector bundle is \emph{very general} if it does not lie on a countable collection of proper closed subschemes of $M(r,d)$.

We begin by showing that on $\PP^1$, stability (balancedness) of vector bundles is preserved by general elementary modifications.

\begin{lem} Suppose a vector bundle $V$ on $\PP^1$ of positive rank is balanced. Then for any $p \in \PP^1$, the kernels of general elementary modifications $V \surj \CC_p$ are balanced.
\end{lem}
\begin{proof} Since $V$ is balanced, $V \simeq \OO(d)^{a} \oplus \OO(d+1)^{b}$ for some $d \in \mathbb{Z}$, $a \ge 0$, and $b > 0$. A general elementary modification $V \surj \CC_p$ induces a surjection $\OO(d+1) \surj \CC_p$ on one of the summands of $V$ of type $\OO(d+1)$, yielding a commutative diagram
\[\xymatrix{
	0 \ar[r] & \OO(d) \ar@{^(->}[d] \ar[r] & \OO(d+1) \ar@{^(->}[d] \ar[r] & \CC_p \ar@{=}[d] \ar[r] & 0 \\
	0 \ar[r] & V' \ar[r] & V \ar[r] & \CC_p \ar[r] & 0
}\]
The exact sequence of cokernels implies that the cokernels of the first two vertical maps are isomorphic, hence there is an exact sequence $\OO(d) \inj V' \surj \OO(d)^a \oplus \OO(d+1)^{b-1}$. But every such extension splits, so $V' \simeq \OO(d)^{a+1} \oplus \OO(d+1)^{b-1}$, which is balanced.
\end{proof}

Next, we prove the key result that allows us to relate Quot schemes of general vector bundles to Quot schemes of trivial bundles.

\begin{prop}\label{p:genbun=genmod} Let $L$ of degree $\ell$ be a sufficiently ample line bundle such that $V^* \otimes L$ is globally generated and has vanishing higher cohomology for all $V \in M(r+s,d)$. Then, up to tensoring by $L^*$, general $V$ and kernels of general elementary modifications $\OO_C^{r+s} \surj \bigoplus_{i=1}^{(r+s)\ell-d} \CC_{q_i}$ coincide.
\end{prop}

\begin{proof} Note that it is possible to find such $L$ because the Castelnuovo-Mumford regularity is bounded on $M(r+s,d)$. Let $N = (r+s)\ell-d$. The moduli space parametrizing all elementary modifications is $\Quot\big(\OO_C^{r+s},(0,N)\big)$, which is irreducible, smooth, and of the expected dimension $N(r+s)$ since $\ext^1(E,F) = 0$ whenever $E$ is locally free and $F$ is torsion. The universal sequence $0 \to \mathcal{E} \to \OO^{r+s} \to \mathcal{F} \to 0$ on  $C \times \Quot\big(\OO_C^{r+s},(0,N)\big)$ assembles the elementary modifications. Given any stable vector bundle $V$, choosing $r+s$ general sections of $V^* \otimes L$ yields a sequence
\[
	0 \to V \otimes L^* \to \OO_C^{r+s} \to T \to 0,
\]
where $T$ is torsion of length $N$. In particular, this sequence  occurs in the universal family over the Quot scheme. Since stability is an open condition in families, this guarantees that general elementary modifications parametrized by the universal sequence produce stable kernels, so by the universal property of the moduli space of sheaves, there is a dominant rational map
\[
	\Quot\left(\OO_C^{r+s},(0,N)\right) \dashrightarrow M(r+s,d),
\]
which completes the proof.
\end{proof}


\subsection{Properties of Quot schemes of general bundles}

By embedding Quot schemes of general bundles in Quot schemes of trivial bundles, we can show that the former inherit the nice properties of the latter.

\begin{prop}\label{p:quot_general} Let $V$ be very general. Then for all $e$:
\begin{enumerate}[(a)]
\item $Q_{e,V}$ is equidimensional of the expected dimension $rd + (r+s)e - rs(g-1)$;
\item $Q_{e,V}$ is generically smooth and the subscheme $U_{e,V}$ of torsion-free quotients in $Q_{e,V}$ is open and dense;
\item Intersections $W$ of general Schubert varieties are proper in each component of $Q_{e,V}$, $W \cap U_{e,V}$ is dense in $W$, and top intersections $W$ are finite, reduced, and contained in $U_{e,V}$.

\end{enumerate}
\end{prop}

Parts (a) and (b) of the proposition for $g \ge 2$ were proved in \cite{Hol04} using other methods.

\begin{proof}
For each $e$, choose a line bundle $L_e$ of degree $\ell_e$ sufficiently ample such that $e + r \ell_e$ is sufficiently large to ensure $Q_{e+r\ell_e}$ has the properties in Proposition \ref{p:quot_trivial_props} and also that $V^* \otimes L_e$ is globally generated with vanishing higher cohomology for all $V \in M(r+s,d)$. General elementary modifications
\[
	\OO_C^{r+s} \surj \bigoplus_{i=1}^{(r+s)\ell_e-d} \CC_{q_i}
\]
produce kernels which, when twisted by $L_e$, are general in $M(r+s,d)$. Thus choosing $V$ very general ensures that for every $e$, there is a sequence
\[
	0 \to V \otimes L_e^* \to \OO_C^{r+s} \to \bigoplus_{i=1}^{(r+s)\ell_e-d} \CC_{q_i} \to 0
\]
in which the elementary modification is general. Now, there are embeddings
\[
	Q_{e,V} \inj Q_{e+r\ell_e,C}, \quad [E \subset V] \mapsto [E \otimes L_e^* \subset V \otimes L_e^* \subset \OO_C^{r+s}]
\]
and the image consists of those subsheaves of $\OO_C^{r+s}$ whose map to the skyscraper sheaf in the elementary modification is zero, namely the image is an intersection
\[
	\bigcap_{i=1}^{(r+s)\ell_e-d} \overline{W}_{1^r}(q_i).
\]
Since the elementary modification is general, so are flags defining the $\overline{W}_{1^r}(q_i)$. Since $Q_{e+r\ell_e,C}$ has the expected dimension and proper Schubert intersections,
\[
	\dim Q_{e,V} = \dim Q_{e+r\ell_e,C} - ((r+s)\ell_e-d)r = (r+s)e + rd - rs(g-1),
\]
which is the expected dimension of $Q_{e,V}$ and proves (a).

For (b), note that quotients in $Q_{e,V}$ (twisted by $L^*$) are obtained from quotients in $Q_{e+r\ell_e,C}$ by elementary modification along the same $\bigoplus \CC_{q_i}$. If the latter quotients are torsion-free, so are the former, which proves that the intersection of $U_{e+r\ell_e,C}$ with the image of the embedding is contained in $U_{e,V}$. Now we get (b) from the same properties for $Q_{e+r\ell_e,C}$ and the fact that the intersection of the $\overline{W}_{1^r}(q_i)$ with $U_{e+r\ell_e,C}$ is dense in the image of the embedding.

For (c), we note that any intersection of Schubert varieties on $Q_{e,V}$ at points other than the $q_i$ can be expressed as the same intersection on $Q_{e+r\ell_e,C}$ together with the additional Schubert varieties $\overline{W}_{1^r}(q_i)$, and since the latter intersection is proper, so is the former. The statement about top intersections also follows immediately from the same statement for $Q_{e,V}$.
\end{proof}

The embedding in the proof of the proposition allows us to replace intersection numbers on Quot schemes of general vector bundles by intersection numbers on Quot schemes of trivial bundles.
\begin{cor}\label{c:gentotriv} Let $V$ be very general of rank $r+s$ and degree $d$. Then for all $e$ and all $\ell \gg 0$,
\[
	\int_{Q_{e,V}} \bar{\sigma}_{\underline{\vec{a}}} = \int_{Q_{e+r\ell,C}} \bar{\sigma}_{\underline{\vec{a}}} \cup \bar{\sigma}_{1^r}^{(r+s)\ell-d}.
\]
\end{cor}
In particular, Proposition \ref{p:chern_classes} extends to the case when $V$ is general, ensuring that the choice of curve in moduli does not change the intersection theory.

\section{Quot schemes on nodal curves}\label{s:quot_nodal}
To prove the $F(g|d)_m^n$ satisfy the relations of a weighted TQFT, we relate Quot schemes on smooth curves to Quot schemes on nodal curves, where the nodal curve is obtained by degeneration. The first important observation is that Schubert varieties in Quot schemes over a nodal curve can be defined at smooth points $p$ in exactly the same way as they were defined over smooth curves. However, the cohomology class of a Schubert variety now depends on which component of the curve contains $p$.

In this section, we change the previous convention by letting $C$ be a reducible nodal curve with two smooth components $C_1$ and $C_2$ of genus $g_1$ and $g_2$ meeting at a simple node $\nu \in C$. Let  $\iota_i \colon C_i \inj C$ denote the embeddings and let $p_i \in C_i$ denote the points $\iota_i^{-1}(\nu)$ lying over the node.

\subsection{Sheaves on nodal curves}\label{ss:nodal_sheaves}

We review some facts about sheaves on reducible nodal curves described in \cite{Ses82}. If $E$ is a rank $r$ torsion-free sheaf on $C$, then its stalk at the node $\nu$ is of the form $E_\nu \simeq \OO_\nu^{r-a} \oplus m_\nu^a$ for some $0 \le a \le r$, where $m_\nu$ is the ideal sheaf of $\nu$.
\begin{defn} Let $E$ be torsion-free sheaf of rank $r$ on the nodal curve $C$. If $E|_\nu \simeq \OO_\nu^{r-a} \oplus m_\nu^a$, then we say $E$ is \emph{$a$-defective}.
\end{defn}
The following proposition lists some relationships between torsion-free sheaves on $C$ and locally-free sheaves on the components of $C$.

\begin{prop}
\begin{enumerate}[(a)]\label{p:nodalsheaves}
\item Let $E$ be rank $r$ and $a$-defective on $C$. Then $\iota_i^*E \simeq E_i \oplus \CC_{p_i}^a$, where the $E_i$ are vector bundles of rank $r$ on $C_i$ and $\deg E_1 + \deg E_2 = \deg E-a$.
\item If $E_i$ are rank $r$ vector bundles on $C_i$, then ${\iota_1}_* E_1 \oplus {\iota_2}_* E_2$ is $r$-defective on $C$ and $\deg ({\iota_1}_* E_1 \oplus {\iota_2}_* E_2) = \deg E_1 + \deg E_2 + r$.
\item Using the notation in (a), there is a canonical short exact sequence
\[
	0 \to E \to {\iota_1}_* E_1 \oplus {\iota_2}_* E_2 \to E(\nu)/({\iota_1}_* \CC_{p_1}^a \oplus {\iota_2}_* \CC_{p_2}^a) \to 0
\]
in which each map ${\iota_i}_* E_i \to E(\nu)/({\iota_1}_* \CC_{p_1}^a \oplus {\iota_2}_* \CC_{p_2}^a)$ is surjective. The quotient in the sequence is isomorphic to $\CC_\nu^{r-a}$.
\end{enumerate}
\end{prop}

\begin{proof}
(a): We argue as in \cite{Ses82}. Let $L$ be a line bundle on $C$ with a section $\OO_C \to L$ that does not vanish at $\nu$, namely the cokernel is supported away from $\nu$. Pulling back along the $\iota_i$ yields sections $\OO_{C_i} \to \iota_i^* L$ whose cokernels partition the original cokernel. Thus $\deg \iota_1^* L + \deg \iota_2^* L = \deg L$. Moreover, a local analysis shows that $\iota_i^* m_\nu = m_{p_i} \oplus \CC_{p_i}$, so $\deg m_{\nu} = \deg m_{p_1} + \deg m_{p_2} + 1$. These facts can be used to deduce the claim.

(b): The key observation is ${\iota_1}_* m_{p_1} \oplus {\iota_2}_* m_{p_2} = m_{\nu}$. The degree statement is true even if $E_i$ are general coherent sheaves because the length of torsion is preserved by ${\iota_i}_*$.

(c) As is clear on the level of modules, there are canonical maps $E \to {\iota_1}_* \iota_1^* E \oplus {\iota_2}_* \iota_2^* E$
that are isomorphisms away from $\nu$.
The functors ${\iota_i}_* \iota_i^*$ and these canonical maps yield a commutative diagram
\[\xymatrix{
	0 \ar[r] & E \ar@{->>}[d]^f \ar[r] & {\iota_1}_* \iota_1^* E \oplus {\iota_2}_* \iota_2^* E \ar@{->>}[d]^-{({\iota_i}_* \iota_i^* f)} \ar[r] & \CC_\nu^{r+a} \ar[d] \ar[r] & 0 \\
	0 \ar[r] & E(\nu) \ar[r]  & {\iota_1}_* \iota_1^* E(\nu) \oplus {\iota_2}_* \iota_2^* E(\nu) \ar[r] & E(\nu) \ar[r] & 0
}\]
in which the first map in the first row is an inclusion since it is an isomorphism away from $\nu$ and $E$ is torsion-free. We identify the cokernel in the first row by its degree. The first map in the bottom row is diagonal inclusion $E(\nu) \to E(\nu) \oplus E(\nu)$. The right vertical map must be a surjection, hence an isomorphism. Removing torsion in the pullbacks yields a map $E \to {\iota_1}_* E_1 \oplus {\iota_2}_* E_2$ that is also injective (for degree reasons), which produces the desired sequence. 
The surjectivity claim follows from the commutativity of the right side of the diagram and the fact that in the bottom row, each $E(\nu)$-summand in the middle surjects onto the quotient.
\end{proof}

Let $V$ be a vector bundle of rank $r+s$ on $C$. Letting $V_i$ denote $\iota_i^* V$, the sequence (c) in the proposition is
\[
	0 \to V \to {\iota_1}_* V_1 \oplus {\iota_2}_* V_2 \to V(\nu) \to 0.
\]
The maps ${\iota_i}_* V_i \surj V(\nu)$ are push-forwards of the quotients $V_i \surj V_i(p_i)$. If we think of the sequence in reverse, starting with $V_1$ and $V_2$, then we are constructing a vector bundle $V$ on the nodal curve by gluing $V_1$ and $V_2$ along the node via a choice of isomorphism $V_1(p_1) \simeq V_2(p_2)$ of their fibers over the node. We can summarize this as a bijection
\[
	\left\{\, \parbox{11em}{\centering vector bundles $V$ on $C$ of rank $r+s$ and degree $d$} \,\right\} \leftrightarrow \left\{\, \parbox{18em}{\centering vector bundles $V_i$ on $C_i$ of rank $r+s$ and degree $d_i$ satisfying $d_1 + d_2 = d$, together with an isomorphism $V_1(p_1) \simeq V_2(p_2)$} \,\right\}.
\]
The map from left to right is $V \mapsto (\iota_1^* V, \iota_2^* V)$ together with the canonical isomorphisms $(\iota_1^* V)(p_1) \simeq V(\nu) \simeq (\iota_2^* V)(p_2)$ coming from the fact that fibers are defined as a pullback, hence are preserved under pullback. The map from right to left is $(V_1,V_2) \mapsto \ker ({\iota_1}_* V_1 \oplus {\iota_2}_* V_2 \surj {\iota_2}_* V_2(p_2))$, where the map ${\iota_1}_* V_1 \to {\iota_2}_* V_2(p_2)$ is the composition of the natural surjection onto the fiber composed with the isomorphism $V_1(p_1) \simeq V_2(p_2)$.

Given a short exact sequence $0 \to E \to V \to F \to 0$ in $Q_{e,V}$, where $E$ is torsion-free but $F$ may have torsion, the natural pull-push maps as in (c) in the proposition fit in a commutative diagram
\[\xymatrix{
	0 \ar[r] & E \ar@{^(->}[d] \ar[r] & V \ar@{^(->}[d]\ar[r] & F \ar[d] \ar[r] & 0 \\
	0 \ar[r] & E_1 \oplus E_2 \ar[r] & V_1 \oplus V_2 \ar[r] & F_1 \oplus F_2 \ar[r] & 0
}\]
in which the rows are exact, $F_i = \iota_i^* F$, $E_i = \iota_i^* E/\text{torsion}$, and the push-forward notation is suppressed in the second row. The snake lemma long exact sequence is of the form
\[
	0 \to \CC_\nu^b \to \CC_\nu^{r-a} \to \CC_\nu^{r+s} \to \CC_\nu^{s+a+b } \to 0
\]
for some $a \ge 0$ and $0 \le b \le r-a$ that can be interpreted as follows: $E$ is $a$-defective and the rank of $E(\nu) \to V(\nu)$ is $r-a-b$. Note that $0 \to E_i \to V_i \to F_i \to 0$ are short exact sequences on $C_i$ and that $\deg E_1 + \deg E_2 = \deg E - a$.

\subsection{Structure of Quot schemes on nodal curves}

To control intersections of Schubert varieties in Quot schemes $Q_{e,V} = \Quot\big( (r,-e),V \big)$ over the nodal curve $C$, we relate $Q_{e,V}$ to Quot schemes $Q_{e_i,V_i}$ over the components $C_i$.

For all $0 \le a \le r$, let $Z_{a,\nu,e}$ denote the locally-closed subscheme in $Q_{e,V}$ consisting of sequences $0 \to E \to V \to F \to 0$ where $E$ is $a$-defective. As in \S\ref{ss:boundary}, let $Z_{a,p_i}$ denote the locally-closed subscheme in $Q_{e_i,V_i} = \Quot\big( (r,-e_i),V_i \big)$ where the subsheaf drops rank by exactly $a$ at $p_i$. Setting $G_a = \Gr\big(r-a,V_1(p_1)\big)$ (which is identified with $\Gr(r-a,V_2(p_2))$, there is an evaluation map
\[
	Z_{a,p_1} \times Z_{a,p_2} \to G_a \times G_a
\]
and we let $\Delta_a$ denote the preimage of the diagonal.

\begin{lem} \label{p:cover} Let $a \ge 0$, $e_1 + e_2 = e + a$, and $0 < b \le r-a$.
\begin{enumerate}[(a)]
\item There is an embedding
\[
	\phi_a \colon \Delta_a \inj Z_{a,\nu,e}
\]
whose image is exactly those $[E \subset V]$ where $E$ is $a$-defective, the map $E(\nu) \to V(\nu)$ has rank exactly $r-a$, and the pullbacks $E_i = (\iota_i^*E)/\text{torsion}$ have degree $e_i$.

\item Let $\mathcal{E}|_\nu$ denote the universal subsheaf restricted to $\{\nu\} \times Z_{a+b,\nu,e-b}$. Let $U_b \subset \Gr(\mathcal{E}|_\nu,b)$ denote the open subscheme of elementary modifications that produce an $a$-defective kernel. There is a map
\[\xymatrix{
	U_{b} \ar[d] \ar[r]^-{\beta_{a,b}} & Z_{a,\nu,e} \\
	Z_{a+b,\nu,e-b}
}\]
whose image contains all $0 \to E \to V \to F \to 0$ where $E$ is $a$-defective and $E(\nu) \to V(\nu)$ has rank $r-a-b$.

\item For each $a$, as $b>0$ and the partition $e_1+e_2$ of  $e+a$ vary, the images of the maps $\phi_a$ and $\beta_{a,b}$ cover $Z_{a,\nu,e}$.
\end{enumerate}
\end{lem}

\begin{proof}
(a): Set-theoretically, define $\phi_a$ as follows. A point in $\Delta_a$ is a pair of inclusions $E_i \subset V_i$ whose restrictions to fibers at $p_i$ have rank $r-a$ and the same image under the identification $V_1(p_1) \simeq V_2(p_2)$. Thus there is a canonical sequence
\[
	0 \to E \to {\iota_1}_* E_1 \oplus {\iota_2}_* E_2 \to \CC_\nu^{r-a} \to 0
\]
producing a sheaf $E$ on $C$ that is $a$-defective. Moreover, the image of $E(\nu) \to V(\nu)$ coincides with the image of $E_i(p_i) \to V_i(p_i)$ after identifying $V_i(p_i)$ with $V(\nu)$.

To construct this morphism algebraically, we construct the same sequence in families. First, we construct vector bundles $\widehat{\mathcal{E}}_i$ on $C \times \Delta_a$ as follows: restrict the universal subsheaves $\mathcal{E}_i \subset \pi_{C_i}^* V_i$ on $C_i \times Q_{e_i,V_i}$ to $C_i \times Z_{a,p_i}$, push forward under the inclusion $\iota_i \times \mathrm{id}$ to get a sheaf on $C \times Z_{a,p_i}$, pull back to $C \times Z_{a,p_1} \times Z_{a,p_2}$ under the projection map, and finally restrict to $C \times \Delta_a$. There is a morphism
\[
	\widehat{\mathcal{E}}_1 \oplus \widehat{\mathcal{E}}_2 \to \pi_C^* V_1(p_1)
\]
whose kernel $\mathcal{E}$ is a flat family of $a$-defective sheaves on $C$ of degree $e$, hence yielding a map to $Q_{e,V}$ whose image is contained in $Z_{a,\nu,e}$.

Similarly, to get the inverse, we pull back the universal subsheaf $\mathcal{E} \to \pi^* V$ on $C \times \im(\phi_a)$ to $C_i \times \im(\phi_a)$ and remove the torsion in $\iota_i^* \mathcal{E}$. Then $\iota_i^* \mathcal{E}/\text{torsion}$ is a flat family of sheaves of degree $e_i$, hence defines a map $\im(\phi_a) \to Z_{a,p_i}$. The image of the product of these maps is contained in $\Delta_a$ since the restrictions to $p_i$ of the inclusions $\iota_i^* \mathcal{E}/\text{torsion} \inj \pi^* V_i$ have image coinciding with the image of $\mathcal{E}|_\nu \to \pi^* V(\nu)$ under the identification $V_i(p_i) \simeq V(\nu)$.

(b): The open subscheme $U_b$ of the Grassmann bundle parametrizes elementary modifications $0 \to E \to E' \to \CC_\nu^b \to 0$. Since $E'$ is $(a+b)$-defective and the elementary modification is general, a local computation shows that $E$, the image under $\beta_{a,b}$, is $a$-defective. Moreover, $E(\nu) \to V(\nu)$ factors through $E'(\nu) \to V(\nu)$, which has rank $\le r - a - b$.

Conversely, starting with $0 \to E \to V \to F \to 0$ in which $E$ is $a$-defective and $E(\nu) \to V(\nu)$ has rank exactly $r-a-b$, by the discussion at the end of \S\ref{ss:nodal_sheaves}, there is a commutative diagram
\[\xymatrix{
	0 \ar[r] & E \ar@{^(->}[d] \ar[r] & V \ar@{^(->}[d] \ar[r] & F \ar[d] \ar[r] & 0 \\
	0 \ar[r] & E_1 \oplus E_2 \ar[r] & V_1 \oplus V_2 \ar[r] & F_1 \oplus F_2 \ar[r] & 0
}\]
and the snake lemma long exact sequence is $0 \to \CC_\nu^b \to \CC_\nu^{r-a} \to \CC_\nu^{r+s} \to \CC_\nu^{s+a+b} \to 0$. The rank of the map of fibers $E_i(p_i) \to V_i(p_i)$ is only $r-a-b$, so we can pass to the diagram
\[\xymatrix{
	0 \ar[r] & E_1 \oplus E_2 \ar@{->>}[d] \ar[r] & V_1 \oplus V_2 \ar@{->>}[d] \ar[r] & F_1 \oplus F_2 \ar@{->>}[d] \ar[r] & 0 \\
	0 \ar[r] & \CC_\nu^{r-a-b} \ar[r] & \CC_\nu^{r+s} \ar[r] & \CC_\nu^{s+a+b} \ar[r] & 0
}\]
and in the short exact sequence of kernels $0 \to E' \to V \to F' \to 0$ the sheaf $E'$ is $(a+b)$-defective. Moreover, there is an induced short exact sequence $0 \to E \to E' \to \CC_\nu^b \to 0$, so $0 \to E \to V \to F \to 0$ is in the image of $\beta_{a,b}$.

Now we construct $\beta_{a,b}$. The restriction of the universal sequence to $C \times Z_{a+b,\nu,e-b}$ yields $0 \to \mathcal{E} \to \pi_C^* V \to \mathcal{F} \to 0$ in which each fiber of $\mathcal{E}$ over $Z_{a+b,\nu,e-b}$ is $(a+b)$-defective, hence $\mathcal{E}|_\nu$ is a vector bundle of rank $r+a+b$ on $Z_{a+b,\nu,e-b}$. Pushing forward the universal sequence $0 \to \mathcal{S} \to \pi^* \mathcal{E}|_\nu \to \mathcal{Q} \to 0$ on $\Gr(\mathcal{E}|_\nu,b)$ under the inclusion $\Gr(\mathcal{E}|_\nu,b) \to C \times \Gr(\mathcal{E}|_\nu,b)$ at the point $\nu$, we get a map $\pi^* \mathcal{E} \surj \mathcal{Q}$ on $C \times \Gr(\mathcal{E}|_\nu,b)$ whose kernel $\mathcal{K}$ is a flat family of degree $e$ sheaves on $C$. The embedding $\mathcal{K} \inj \pi^* \mathcal{E} \inj \pi^* V$ induces the map $\beta_{a,b} \colon \Gr(\mathcal{E}|_\nu,b) \to Q_{e,V}$. Restricting to $U_b$ ensures that the image is contained in $Z_{a,\nu,e}$.

(c): Let $x=[0 \to E \to V \to F \to 0]$ be a short exact sequence in which $E$ is $a$-defective and $E(\nu) \to V(\nu)$ has rank $r-a-b$. If $b > 0$, then  $x$ is in the image of $\phi_a$ (for $e_i = \deg E_i$). If $b > 0$, then $x$ is in the image of $\beta_{a,b}$.
\end{proof}

\begin{rmk} In the case when the $V_i$ are very general and the isomorphism $V_1(p_1) \simeq V_2(p_2)$ defining $V$ is sufficiently general to ensure that $\Delta_0$ always has the expected dimension, it follows from (a) that for fixed $e$, each partition $e = e_1 + e_2$ for which $\Delta_a$ is nonempty yields a component of $Q_{e,V}$ of the expected dimension. By (c) and a dimension count, we see that any other component of $Q_{e,V}$ must have smaller dimension. Thus, although, $Q_{e,V}$ is in general not irreducible, each of its top-dimensional components is the pullback of the diagonal in a product $U_{e_1,V_1} \times U_{e_2,V_2}$ for some partition $e_1 + e_2 = e$. In particular, when $V$ is trivial, we see that each component of $\Mor_e(C,G)$ is the pullback of the diagonal in $\Mor_{e_1}(C_1,G) \times \Mor_{e_2}(C_2,G)$, namely a morphism $C \to G$ of degree $e$ corresponds to a pair of morphisms $C_i \to G$ of degrees summing to $e$ such that the points $p_i$ have the same image in $G$.
\end{rmk}

\subsection{Intersections of Schubert varieties are proper}

Using the maps described in the previous subsection, we study intersections of Schubert varieties in $Q_{e,V}$ by controlling their preimages. We can use induction for the maps $\beta_{a,b}$, so the main tool we still need is a properness statement for intersections of Schubert varieties in $\Delta_a \subset Z_{a,p_1} \times Z_{a,p_2}$. In the following lemma, since $0 \le a \le r$ and the $p_i$ are fixed, we drop $a$, $p_1$, and $p_2$ from the notation, so we write $Z_{e_i}:=Z_{a,p_i}$ since we do need to keep track of the degrees $e_i$.

\begin{lem} Suppose $\Delta \subset Z_{e_1} \times Z_{e_2}$ is empty or has pure codimension $(r-a)(s+a)$ for all $e_1,e_2$. Then intersections $W$ of general Schubert varieties on $Q_{e_1,V_1} \times Q_{e_2,V_2}$ are proper on $\Delta$. Moreover, if $\mathcal{U} \subset Z_{e_1} \times Z_{e_2}$ denotes the open locus where the quotients are torsion-free away from $p_1$ and $p_2$, then $W \cap \Delta \cap \mathcal{U}$ is dense in $W \cap \Delta$. In particular, top intersections $W$ (where the length of the partitions equals $\dim \Delta$) are finite, reduced, and contained in $\mathcal{U} \cap \Delta$.
\end{lem}

\begin{proof} The proof adapts Strategy \ref{strategy:proper} to this more complicated setting. Let $W$ be the intersection of $\overline{W}_{\vec{a}_1}(q_1),\dots,\overline{W}_{\vec{a}_{k}}(q_k)$ and $\overline{W}_{\vec{b}_1}(q_1'),\dots,\overline{W}_{\vec{b}_m}(q_m')$, where $q_1,\dots,q_k \in (C_1 \setminus p_1)$ and $q_1',\dots,q_m' \in (C_2 \setminus p_2)$ are distinct points. Set $A = \sum_{i=1}^k |\vec{a}_i| + \sum_{j=1}^m |\vec{b}_i|$. Let $\mathcal{B}$ denote the complement of $\mathcal{U}$ in $Z_{e_1} \times Z_{e_2}$. Since the Schubert varieties on $\Delta$ are still defined as degeneracy loci of the restriction of the relevant bundles on $Z_{e_1} \times Z_{e_2}$, their codimension in $\Delta$ cannot be strictly larger than $A$. Thus it suffices to prove that
\begin{enumerate}[(1)]
\item $W \cap \Delta \cap \mathcal{U}$ is empty or has pure codimension $A$ in $\Delta$;
\item $W \cap \Delta \cap \mathcal{B}$ is empty or has codimension $> A$ in $\Delta$.
\end{enumerate}

(1): As usual, there is an evaluation morphism $\mathrm{ev}_{q_1,\dots,q_k,q_1',\dots,q_m'} \colon \mathcal{U} \to G^{k+m}$. Restricting the domain to $\Delta \cap \mathcal{U}$ and using Lemma \ref{l:morphism_proper} gives the result.

(2): The boundary $\mathcal{B}$ is the image of the map
\[
	\beta_{1,0} \sqcup \beta_{0,1} \colon \big(\Gr(\mathcal{E}_{e_1-1},1) \times Z_{e_2} \big) \sqcup \big(Z_{e_1} \times \Gr(\mathcal{E}_{e_2-1},1)\big) \to Z_{e_1} \times Z_{e_2},
\]
where the definition of Grassmann bundles over $Z_{e_1-\ell}$ is as before, except that we restrict the elementary modification to a point of $C_1 \setminus p_1$, thus obtaining a recursive structure on the $Z_{e_1-\ell}$ for varying $\ell$ (and similar for $Z_{e_2-\ell}$).

We pull back $\Delta$ under $\beta_{1,0}$ and $\beta_{0,1}$ to get closed subvarieties $\Delta_{1,0}$ and $\Delta_{0,1}$ in the Grassmann bundles. As usual, it suffices to prove that the pullback of $W$ has codimension
\[
	> A - (\dim \Delta - \dim \Delta_{1,0})
\]
in $\Delta_{1,0}$ (then, by symmetry, we also get codimension $> A - (\dim \Delta - \dim \Delta_{0,1})$ in $\Delta_{0,1}$). Since the elementary modifications in $\Gr(\mathcal{E}_{e_1-1},1)$ occur at points other than $p_1$, there is a commutative diagram
\[\xymatrix{
	\Gr(\mathcal{E}_{e_1-1},1) \times Z_{e_2} \ar[d]_{\pi_{1,0}} \ar[r]^{\beta_{1,0}} & Z_{e_1} \times Z_{e_2} \ar[d]^{\ev_{p_1,p_2}} \\
	(C \setminus p) \times Z_{e_1-1} \times Z_{e_2} \ar[r]_-{\ev_{p_1,p_2}'} & \Gr(a,\CC^{r+s}) \times \Gr(a,\CC^{r+s})
}\]
so $\Delta_{1,0}$ can be obtained by pulling back the diagonal under $\ev'_{p_1,p_2}$ and then $\pi_{1,0}$. By assumption, the pullback under $\ev'_{p_1,p_2}$ is empty or has codimension $(r-a)(s+a)$, hence we get that same codimension in the Grassmann bundle. Thus
\[
	\dim \Delta - \dim \Delta_{1,0} = \dim (Z_{e_1} \times Z_{e_2}) - \dim (\Gr(\mathcal{E}_{e_1-1},1) \times Z_{e_2}) = s.
\]

We prove that the intersection $W \cap \Delta_{1,0}$ is empty or has codimension $> A - s$ in $\Delta_{1,0}$ by induction on $e_1 + e_2$. The base case is trivial since $Z_{e_1} \times Z_{e_2}$ is empty when $e_1 + e_2$ is sufficiently small. For the inductive step, we may assume there is at least one Schubert variety since otherwise there is nothing to prove. We can immediately deal with the case $k=0$ and $m > 0$ since by the inductive hypothesis the Schubert varieties impose the right codimension on $Z_{e_1-1} \times Z_{e_2}$, hence also on $\Gr(\mathcal{E}_{e_1-1},1) \times Z_{e_2}$.

The remaining case is $k > 0$. As usual, the intersection on $\Gr(\mathcal{E}_{e_1-1},1)$ decomposes into two types of components. The Type 1 component reduces to an intersection on the base $Z_{e_1-1}$, where we are done by induction. Up to relabeling of the $q_i$, each Type 2 component is supported on $\Gr(1,q_k)=\pi_1^{-1}(\{q_k\} \times Z_{e_1-1})$ and includes the degeneracy locus $\hat{W}_{(\vec{a}_k)_{\ell+1},\dots, (\vec{a}_k)_r}(q_k)$. As described in Corollary{c:stratification}, there is a stratification
\[
	\Gr(1,q_k) = \bigsqcup_{1 \le \ell \le r} \beta_{1,q_k}^{-1}(Z_{\ell,q_k}),
\]
where $U_{\ell,q_k}$ is the open subscheme of the Grassmann bundle $\Gr(\ell,q_k)$ over $Z_{e_1-\ell}$ and $Z_{\ell,q_k} = \beta_{\ell,q_k}(U_{\ell,q_k}) \subset Z_{e_1}$. The pullback of $\Delta$ (which we also write as $\Delta$) intersects each $U_{\ell,q_k} \times Z_{e_2}$ properly since it is proper on the base $Z_{e_1-\ell} \times Z_2$. The preimage of $\beta_{1,q_k}$ over $Z_{\ell,q_k}$ has fibers of dimension $\ell-1$. There are evaluation maps $\mathrm{ev}_{q_k} \colon U_{\ell,q_k} \to \Gr\big(r-\ell,V_1(q_k)\big)$ yielding maps
\[
	U_{\ell,p_k} \times Z_{e_2} \to \Gr\big(r-\ell,V_1(q_k)\big).
\]
The inductive assumption ensures that the intersection of $\Delta$ and the Schubert varieties is proper in $Z_{e_1-\ell} \times Z_{e_2}$, hence also in $U_{\ell,q_k} \times Z_{e_2}$. A general choice of $W_{\vec{a}_k}$ in $G$ ensures that each $W_{(\vec{a}_k)_{\ell+1},\dots, (\vec{a}_k)_r}$ is general in $\Gr\big(r-\ell,V_1(q_k)\big)$, so the Type 2 component imposes codimension $\ge A - \ell s$ in $\Delta \cap (U_{\ell,q_k} \times Z_{e_2})$, hence also in $\Delta \cap (\beta_{1,q_k}^{-1}(Z_{\ell,q_k}) \times Z_{e_2})$. Thus by (\ref{eq:grassmann_strat_dim}), the preimage under $\beta_{1,q_k}$ of the intersection of $\Delta$ and the Schubert varieties has codimension
\[
	\ge A - \ell s + \big(\dim \Gr(1,q_k) - \dim \Gr(\ell,q_k)\big) - (\ell-1) = A - s + \ell(\ell-1)
\]
in $\Delta \cap (\Gr(1,q_k) \times Z_{e_2})$, and passing to $\Delta \cap (\Gr(\mathcal{E}_{e_1-1},1) \times Z_2)$ yields one additional codimension. The largest such locus is obtained when $\ell = 1$, but this still has codimension $\ge A - s + 1$. This completes the case $k > 0$.
\end{proof}

With this technical tool in hand, we can prove our result. Recall that $V$ is determined by specifying an isomorphism $V_1(p_1) \simeq V_2(p_2)$. We call $V$ \emph{very general} if the $V_i$ are very general (as in \S\ref{s:quot_general}) and this isomorphism is very general. This guarantees that the properness assumption on $\Delta$ in the previous lemma is satisfied.

\begin{prop}\label{p:quot_nodal} Suppose $V$ is very general. Let $\tilde{Q}_{e,V}$ denote the union of the top-dimensional components of $Q_{e,V}$. Then for all $e \in \mathbb{Z}$,
\begin{enumerate}[(a)]
\item $\tilde{Q}_{e,V}$ is generically smooth and of the expected dimension $rd + (r+s)e - rs(g-1)$ (or empty if this is negative);
\item The subscheme $U_{e,V}$ of torsion-free quotients in $Q_{e,V}$ is contained in $\tilde{Q}_{e,V}$ and is open and dense;
\item Intersections $W$ of general Schubert varieties are proper in each component of $\tilde{Q}_{e,V}$;
\item Top intersections $W$ are finite, reduced, and contained in $U_{e,V}$.
\end{enumerate}
\end{prop}

\begin{proof} We cover $Q_{e,V}$ with the images of maps $\phi_a$ and $\beta_{a,b}$ from Proposition \ref{p:cover} for all $0 \le a \le r$, $0 < b \le r-a$, and $e_1 + e_2 = e + a$. The maps $\phi_0$ for various partitions $e_1+e_2=e$ identify each component of $U_{e,V}$ as the diagonal $\Delta_0$ in some $U_{e_1,V_1} \times U_{e_2,V_2}$. Since the $U_{e_i,V_i}$ are all generically-smooth of the expected dimension, and since the domains of the other maps have smaller dimension, we get (a) and (b).

For (c), it suffices to prove that the preimages of Schubert varieties in $Q_{e,V}$ have failure 0 in the domain of each $\phi_a$ and $\beta_{a,b}$. For $\phi_a$ this follows from the previous lemma. For the maps $\beta_{a,b}$, we can perform the intersection on the base, where it is proper by induction on $e$.

It follows from (c) that top intersections can be performed by pulling back to the diagonal in each $U_{e_1,V_1} \times U_{e_2,V_2}$, so Lemma \ref{l:morphism_proper} ensures (d).
\end{proof}

\begin{rmk} Although the proposition does not rule out the possibility of additional low-dimensional components in the Quot scheme, it does ensure that such components do not affect the intersection numbers on $Q_{e,V}$.
\end{rmk}

\begin{cor}\label{c:degeneration_formula} Suppose $V$ is very general. Let $\bar{\sigma}_{\underline{\vec{a}}_i}$ denote a cup product of Schubert cycles based at points on the component $C_i$. Then for all $e$,
\[
	\int_{Q_{e,V}} \bar{\sigma}_{\underline{\vec{a}}_1} \cup \bar{\sigma}_{\underline{\vec{a}}_2}
	= \sum_{e_1 + e_2 = e} \sum_{\vec{b}} \left( \int_{Q_{e_1,V_1}}\bar{\sigma}_{\underline{\vec{a}}_1} \cup \bar{\sigma}_{\vec{b}} \right) \left( \int_{Q_{e_2,V_2}}\bar{\sigma}_{\underline{\vec{a}}_2} \cup \bar{\sigma}_{\vec{b}^c} \right).
\]
\end{cor}

\begin{proof} By the previous proposition, top intersections are contained in $U_{e,V}$. Since the top-dimensional components of $U_{e,V}$ are isomorphic to $(U_{e_1,V_1} \times U_{e_2,V_2}) \cap \Delta_0$ for partitions $e_1 + e_2 = e$, we can compute the intersection on each $(U_{e_1,V_1} \times U_{e_2,V_2}) \cap \Delta_0$. Since the class of the diagonal in $G \times G$ is $\sum_{\vec{b}} \sigma_{\vec{b}} \otimes \sigma_{\vec{b}^c}$, the class of $\Delta_0$ in $Q_{e_1,V_1} \times Q_{e_2,V_2}$ is $\sum_{\vec{b}} \bar{\sigma}_{\vec{b}} \otimes \bar{\sigma}_{\vec{b}^c}$ (up to a class supported on the boundary, which has no effect on intersection numbers). Now the formula follows by pairing this class with the Schubert cycles in cohomology.
\end{proof}
The right side of the equation in the corollary is independent of the components $C_i$ in moduli and the points $p_i \in C_i$, so a version of Proposition \ref{p:chern_classes} holds for nodal curves as well. 

Since the $V_i$ are very general, we can express the integrals on the smooth components in the corollary as integrals on Quot schemes of trivial bundles for very large $e$. We then apply the corollary again to get an integral on $Q_{e,C}$ (though the corollary is only stated for $V$ very general, it also holds for sufficiently large $e$ when the $V_i$ are trivial bundles since the $Q_{e_i,C_i}$ have the right properties and the gluing of the trivial bundles can still be chosen very general). Thus, as in the case of smooth curves, all integrals on Quot schemes of very general vector bundles can be expressed as integrals on Quot schemes of trivial bundles.

\begin{cor}\label{c:nodal_gentotriv} In the setting of the previous corollary, let $\deg V_i = d_i$ with $d_1+d_2=d$. Then letting $\ell = \ell_1+\ell_2$ for $\ell_i$ sufficiently large,
\[
	\int_{Q_{e,V}} \bar{\sigma}_{\underline{\vec{a}}_1} \cup \bar{\sigma}_{\underline{\vec{a}}_2}
	= \int_{Q_{e+r\ell,C}} \big(\bar{\sigma}_{\underline{\vec{a}}_1} \cup \bar{\sigma}_{1^r}^{(r+s)\ell_1-d_1}\big) \cup \big( \bar{\sigma}_{\underline{\vec{a}}_2} \cup \bar{\sigma}_{1^r}^{(r+s)\ell_2-d_2} \big).
\]
\end{cor}
This simplifies the deformation argument in the next section by avoiding the need to consider deformations of general bundles.

\subsection{Proof of Main Theorem (a)}

Given the formula in Corollary \ref{c:degeneration_formula} relating intersection numbers on nodal curves and their components, the last step is to relate intersection numbers on nodal curves with intersection numbers on smooth curves of the same genus. Let $C$ be a nodal curve of arithmetic genus $g$ and $C'$ denote a smooth curve of genus $g$ obtained by smoothing $C$.

\begin{lem}\label{l:nodal_deformation} For all $e \gg 0$, the intersection numbers on $Q_{e,C}$ agree with the intersection numbers on $Q_{e,C'}$.
\end{lem}

\begin{proof} The proof is similar to Proposition 1.5 in \cite{Ber94}. Let $\mathcal{C}$ be a family over a base curve $B$ smoothing $C$, where $\OO_\mathcal{C}^{r+s}$ is the trivial deformation of $\OO_C^{r+s}$. Consider the relative Quot scheme $\pi \colon Q = \Quot\big((r,-e),\OO_{\mathcal{C}}^{r+s},B\big) \to B$, whose fibers over $b \in B$ are $Q_{e,C_b}$ and which is projective over $B$. If $\pi$ is not already flat at the central fiber $b_0 \in B$, then there must be associated points of $Q$ supported at $b_0$; these arise either from nonreducedness or from small components of $Q_{e,C}$ because the top-dimensional components are all of the expected dimension and hence deform in all families by (this follows from Theorem 2.15 of \cite{Kol96}). In any case, we can replace the central fiber by its flat limit without affecting the top-dimensional components.

Now $\pi$ is flat and we want to apply Lemma 1.6 of \cite{Ber94}. After restriction and base change, we can find a section $\sigma$ of $\mathcal{C}$ near each $b \in B$ and restrict the universal subsheaf (which is defined on $\mathcal{C} \times_B Q$) to $\sigma \times_B Q$. The lemma implies that top intersections of Chern classes of the universal subbundles in the fibers are independent of the base point, hence the intersection numbers on the Quot schemes are independent of the base point by Proposition \ref{p:chern_classes}. On the central fiber these intersection numbers compute the intersection numbers on the nodal curve (because the lower-dimensional components do not interfere with top intersections of Schubert varieties).
\end{proof}

The lemma implies that intersection numbers on the nodal curve $C$ do not depend on how the Schubert varieties are spread across the two components of $C$. Moreover, using Corollary \ref{c:nodal_gentotriv}, we deduce that $Q_{e,V}$ and $Q_{e,V'}$ have the same intersection numbers for all $e$ when $V$ and $V'$ be very general vector bundles on $C$ and $C'$ of the same rank $r+s$ and degree $d$.

\begin{cor} Suppose $g = g_1 + g_2$ and $d = d_1 + d_2$. Then
\[
	F(g|d)_m^n = F(g_2|d_2)_1^n \circ F(g_1|d_1)_m^1.
\]
\end{cor}

\begin{proof}
It suffices to prove that the images of $\sigma_{\vec{a}_1} \otimes \cdots \otimes \sigma_{\vec{a}_m}$ under each of the maps have the same coefficient of $\sigma_{\vec{b}_1} \otimes \cdots \otimes \sigma_{\vec{b}_n}$. The map $F(g|d)_m^n$ yields the coefficient
\[
	\int_{Q_V} \bar{\sigma}_{\underline{\vec{a}}} \cup \bar{\sigma}_{\underline{\vec{b}}} \tag{$\dagger$},
\]
where $V$ is a very general vector bundle of degree $d$ on a smooth curve of genus $g$. The composition on the right yields the coefficient
\[
	\sum_{\vec{c}} \left( \int_{Q_{V_1}} \bar{\sigma}_{\underline{\vec{a}}} \cup \bar{\sigma}_{\vec{c}} \right) \left( \int_{Q_{V_2}} \bar{\sigma}_{\underline{\vec{b}}} \cup \bar{\sigma}_{\vec{c}^c} \right), \tag{$\ddagger$}
\]
where the $V_i$ are very general vector bundles of degree $d_i$ on smooth curves of genus $g_i$. By summing the equation in Corollary \ref{c:degeneration_formula} over all $e$, the expression $(\ddagger)$ is equal to an integral on a Quot scheme over the nodal curve obtained by gluing the $C_i$ at a point, and this integral coincides with ($\dagger$) by the lemma.
\end{proof}

It is easy to deduce the general partial composition relations of the weighted TQFT from the corollary by gluing one boundary circle at a time. This completes the proof of Main Theorem (a).

\bibliographystyle{alpha}
\bibliography{My_Library}

\end{document}